\documentclass[11pt]{amsart}
\usepackage{amssymb}
\usepackage{enumerate}

\newcommand{\N}{\mathbb{N}}

\newcommand{\Q}{\mathbb{Q}}

\newcommand{\GAUT}{\overline G^{\mathfrak S}} 
\newcommand{\GEMB}{\overline G} 

\DeclareMathOperator{\Id}{Id}

\DeclareMathOperator{\sib}{sib}
\DeclareMathOperator{\age}{age}
\DeclareMathOperator{\Min}{Min}
\DeclareMathOperator{\Max}{Max}

\def\Power #1 { \powerset(#1) }

\newtheorem{definition}{{\bf Definition}}[section]
\newtheorem{theorem}[definition]{{\bf Theorem}}

\newtheorem{corollary}[definition]{{\bf Corollary}}
\newtheorem{case}[definition]{\noindent {\bf Case}}
\newtheorem{proposition}[definition]{\noindent {\bf Proposition}}
\newtheorem{lemma}[definition]{\noindent {\bf Lemma}}

\newtheorem{claim}[definition]{\noindent {\bf Claim}}
\newtheorem{subclaim}[definition]{\noindent {\bf Subclaim}}
\newtheorem{conjecture}[definition]{\noindent {\bf Conjecture}}

\newtheorem{example}[definition]{\noindent {\bf Example}}

\newtheorem{remark}[definition]{\noindent {\bf Remark}}
\newtheorem{problem}[definition]{\noindent {\bf Problem}}
\newtheorem{problems}[definition]{\noindent {\bf Problems}}

\def\proofref #1 {{\noindent  {\bf Proof} (#1).}\ }
\def\endproof{\hfill {\kern 6pt\penalty 500
\raise -0pt\hbox{\vrule \vbox to5pt {\hrule width 5pt
\vfill\hrule}\vrule}}}
\def\centerpicture #1 by #2 (#3){\leavevmode
        \vbox to #2{
        \hrule width #1 height 0pt depth 0pt
        \vfill
        \special{pictfile #3}}}
\title[Siblings]{Siblings of an $\aleph_0$-categorical relational structure. }

\author[C.Laflamme]{Claude Laflamme} 
\address{ Mathematics \& Statistics Department, University of Calgary, Calgary, Alberta, Canada T2N 1N4}
\email{laflamme@ucalgary.ca} 
\author[M.Pouzet]{Maurice Pouzet} \address{ICJ, Math\'ematiques, Universit\'e Claude-Bernard Lyon1, 43 bd. 11 Novembre 1918, 69622 Villeurbanne Cedex, France and Mathematics \& Statistics Department, University of Calgary, Calgary, Alberta, Canada T2N 1N4}
 \email{pouzet@univ-lyon1.fr }
\author[N.Sauer]{Norbert Sauer} \address{Mathematics \& Statistics Department, University of Calgary, Calgary, Alberta, Canada T2N 1N4}
\email{nsauer@ucalgary.ca }  
\author[R.Woodrow]{Robert Woodrow} \address{Mathematics \& Statistics Department, University of Calgary, Calgary, Alberta, Canada T2N 1N4}
\email{woodrow@ucalgary.ca }  

\thanks{This research started while the second author visited the Mathematics and Statistics Department of the University of 
Calgary in June 2012; the support provided is gratefully acknowledged.  The first, third and fourth authors warmly thank the Logic group and their staff at the \emph{Institut Camille Jordan} of Universit\'e Lyon I for their wonderful hospitality during various visits during the preparation of this work. \\
 This work was supported financially by NSERC of Canada Team Grant \#~10007490, and also  by the LABEX MILYON (ANR-10-LABX-0070) of Universit\'e de Lyon within the program ``Investissements d'Avenir (ANR-11-IDEX-0007)" operated by the French National Research Agency (ANR)}

\date{\today}

\begin{document}


\begin{abstract}  
A \emph{sibling} of a relational structure $R$ is any structure $S$ which can be embedded into $R$ and, vice versa,  in which $R$ can be embedded. Let $\sib(R)$ be the number of siblings of $R$, these siblings being  counted up to isomorphism.  Thomass\'e conjectured that for countable relational structures made of at most countably many relations, $\sib(R)$ is either $1$, countably infinite, or the size of the continuum; but even showing the special case $\sib(R)=1$ or infinite is unsettled when $R$ is a countable tree.

We prove that if $R$ is countable and $\aleph_{0}$-categorical, then indeed $\sib(R)$ is one or infinite. Furthermore, $\sib(R)$ is one if and only if $R$ is finitely partitionable in the sense of Hodkinson and Macpherson \cite{hodkinson-macpherson}.  The key  tools in our proof are the notion of monomorphic decomposition of a relational structure introduced in \cite{pouzet-thiery} and studied further in \cite {oudrar-pouzet},  \cite{oudrar} and a result of Frasnay \cite{ frasnay 84}.\end{abstract} 
\maketitle


Dedicated to  Roland Fra\"{\i}ss\'e and Claude Frasnay. In memoriam.

\section{Introduction} 

A \emph{sibling} of a given relational structure $R$ is any structure $S$ which can be embedded into $R$, and vice versa, in which $R$ can be embedded. 
If $R$ is finite, there is just one sibling but generally one cannot expect equimorphic structures to be necessarily  isomorphic. However,  the famous Cantor-Bernstein-Schroeder Theorem states that this is the case for structures in a language with pure equality: if there is an injection from one set to another and vice-versa, then there is a bijection between these two sets. The same situation occurs in other structures such as vectors spaces, where embeddings are linear injective maps. But, as expected, it is not in general the case that equimorphic structures are isomorphic. 

Thus, let $\sib(R)$ be the number of siblings of $R$, these siblings being  counted up to isomorphism. Thomass\'e conjectured that $sib(R)=1$, $\aleph_0$ or $2^{\aleph_0}$ for countable relational structures made of at most countably many relations  (see Conjecture 2 in  \cite{thomasse}).  We verified this conjecture for chains in \cite{LPW}. The special case,  $\sib(R)=1$ or infinite, is unsettled, even  in the case of trees. It is connected to the Bonato-Tardif conjecture  which asserts that  for every tree $T$ the number of trees which are siblings of $T$ is either one or infinite, see \cite{bonato-tardif, bonato-al, tyomkyn}. The connection is through the following observation.  Every sibling of a tree $T$ is a tree if and and only if $T\oplus 1 $, the graph obtained by adding to $T$  an isolated vertex, is not a sibling of $T$ (more generally, note that every sibling of a connected graph is connected, just in case $G\oplus 1$ is not a sibling). Hence, for a tree $T$ not equimorphic to $T\oplus 1$, the Bonato-Tardif conjecture and the special case of Thomass\'e's conjecture are equivalent. It turns out that for these trees,  these conjectures are open (for an example, it is open for ternary trees decorated with pendant vertices). On the other hand,   if a tree $T$ is equimorphic to $T\oplus 1$, the number of siblings of $T$ is infinite, hence the special case of the Thomass\'e conjecture holds, but we do not know if the Bonato-Tardif conjecture holds).  \medskip

In this paper we prove the following: 

\begin{theorem}\label{thm:main} The  number of siblings of a countable  $\aleph_0$-categorical  relational structure $R$ is either one or infinite. Furthermore,  it is  one if and only if $R$ is finitely partitionable, that is there is a partition of the domain $E$ of $R$ into finitely many sets such that every permutation of $E$ which preserves each block of the partition is an automorphism of $R$.
\end{theorem}

Our  result extends  a result of Hodkinson and Macpherson \cite{hodkinson-macpherson}. Indeed, they  proved that  a countable structure $R$ in a finite  language is such that   every  $R'$ with the same age is isomorphic  to  $R$ (in which case   every $R'$ with the same age is equimorphic to $R$), if and only if  $R$ is finitely partitionable. They  indicate that their result   holds if the language is infinite and, in addition $Aut(R)$, the automorphism of $R$, is oligomorphic,  that is, for each integer $n$,  the number of orbits of $n$-element subsets of the base set is finite.

The fact that a countable relational structure $R$ is $\aleph_0$-categorical is equivalent to the fact that  $Aut (R)$  is oligomorphic (Engeler, Ryll-Nardzewski and Svenonius, see for example Cameron \cite{cameron.1990} p.30). In this context,  our result applies to countable homogeneous structures with an oligomorphic automorphism group.  Indeed, let $G$ be a group acting on a set $E$. We recall that a  partial map $f$ with domain $A$ and codomain $A'$, subsets of $E$, is \emph{adherent} to $G$ w.r.t. the pointwise convergence topology if  for every finite subset $F$ of $A$ there is some $g\in G$ such that $f$ and $g$ coincide on $F$. In our setting, we will instead say that such a map is a $G$-\emph{local embedding}; if $A=E$ then we say that this is a $G$-\emph{embedding},  and if furthermore $A'=E$ we say that this is a $G$-\emph{automorphism}. We write $\GEMB$ for the set of $G$-embeddings, and we write $\GAUT$ for the set of $G$-automorphisms which is easily seen to form a group.  If $G= \GAUT$, we say that $G$ is \emph{closed} (this is the case  if $G= Aut(R)$ for some relational   structure $R$).  We say that two subsets of $E$ are \emph{equivalent}, resp. \emph{weakly-equivalent},  if each is the image of the other by some $G$-local embedding, resp. each one \emph{contains} the image of the other by some $G$-local embedding. A \emph{$G$-copy} is the image of $E$ under some $G$-embedding, that  is,  a member of the equivalence class of $E$. A \emph{$G$-sibling}  is a subset of $E$ which contains a $G$-copy; equivalently, this is a subset weakly equivalent to $E$.  We denote by $\sib(G)$ the number of equivalence classes of $G$-siblings, under isomorphism.

\medskip

In this setting, Theorem \ref{thm:main} yields the following.

\begin{theorem}
If $G$ is a closed oligomorphic group  on a countable set $E$, then $\sib(G)$  is one or infinite. That is  either the weak-equivalence classes  of $E$ coincide with the equivalence classes of $E$ (that is the set of copies), or  each  is the union of  infinitely many equivalence classes. In the first case there is a partition of  $E$ into finitely many sets such that every permutation of $E$ which preserves each block of the partition belongs to  $G$.   
\end{theorem} 

\begin{proof} Since $G$ is closed, there exist some  homogeneous relational structure $R$ such that $Aut(R)= G$ (see for example Cameron \cite{cameron.1990} p.26). Since $G$ and hence $Aut(R)$ is oligomorphic, $R$ is $\aleph_0$-categorical.  This $R$ is such that a partial map is a local embedding of $R$ iff it  is a $G$-local embedding. Hence, the number of equivalence classes of $G$-siblings is exactly the number of siblings of $R$.  
\end{proof}

\medskip

The number of siblings of a countable  $\aleph_0$-categorical structure can be 1 or $\aleph_0$, but our proof does not show if $2^{\aleph_0}$ is the only other possibility.

\subsection{Ideas behind the proof. An  outline}

A natural idea in the study of siblings of a structure $R$ is to study extensions of $R$ with the same age. When $R$ is universal for its age, these extensions are automatically siblings. 

To illustrate, let us consider countable homogeneous graphs. Thanks to the classification result of Lachlan-Woodrow \cite{lachlan-woodrow} we have a precise description. Each such graph is (up to complement) the Random graph (where the  age is all finite graphs); the generic structure whose age is all $K_n$-free graphs ($n\geq 3$);  $mK_n$  (where $m+n$ is infinite, $m,n\geq 1$). 
Using the idea of non-isomorphic extensions, we can easily produce $2^{\aleph_0}$  siblings for $G$, the Random graph,  or $G$, the homogeneous $K_n$-free graph. Indeed, let $\{G_n: n\in \N\}$ be an antichain (for graph embedding) of finite connected graphs without triangles (e.g., take for $G_n$ an $n+4$-element cycle). For $S\subseteq \N$, form $G_S:= G\cup \sum_{n\in S}G_n$, the disjoint union of $G$ and some of the $G_n$. Since $G$ is connected, these graphs are not isomorphic; since $G$ is universal for its age, they are equimorphic to $G$. Hence $\sib(G)=2^{\aleph_0}$.
When $G= mK_n$, three cases need to be considered. Case 1. $m,n$ are infinite, For $S\subseteq \N$, form $G_S= G\cup \sum_{n\in S}K_n$. Clearly, $G_S$ embeds in $G$ and this produces $2^{\aleph_0}$ siblings. Case 2. $m$ finite. In this case, $\sib(G)= 1$. Case 3. $m$ is infinite;  we may suppose $n\geq 2$. In this  case, by extending $G$ to isolated vertices, $\sib(G)= \aleph_0$. It is not difficult to use the same idea to show that for the countable ultra-homogeneous tournaments one has the same trichotomy. It is tempting to try to generalize the results to relational structures $R$ that are universal for their own age. But this goes beyond techniques we have. By restricting the classes to $\aleph_0$-categorical structures, and by using the idea of monomorphic decomposition, one can get some general results showing $\sib(R)$ is $1$ or infinite. 
We sketch the outline of the proof.

We start with a countable structure $R$ which is $\aleph_0$-categorical in its complete theory. As is well known,  there is a   countable structure $R'$ equimorphic to $R$ which is $\aleph_0$-categorical, but for which the complete theory is axiomatizable by universal-existential sentences (see Saracino \cite{saracino}, see also  Pouzet \cite{pouzet}  last theorem of page 697). Since $R'$ is equimorphic to $R$, then $\sib(R')= \sib(R)$, and hence  we may replace $R$ by $R'$.

Structures $R$ for which the complete theory is axiomatizable by universal-existential sentences have a combinatorial definition that we recall in Section \ref{section: basic} (Theorem \ref{thm:unifprehomtest}).  They are  uniformly prehomogeneous and  their profile  (the function which counts for each integer $n$ the number of restrictions  to the $n$-elements subsets,  these restrictions being counted up to isomorphy) take only finite values. 

%
Starting with such a structure $R$,  we consider its  monomorphic decomposition. This notion   appears in full generality in  \cite{pouzet-thiery}, \cite {oudrar-pouzet}, and in \cite{oudrar}). In our case it is given by an equivalence relation that is definable by a universal sentence.  

We study first a special case,when  the decomposition consists of one class, that is, in the terminology of Fra\"{\i}ss\'e,  $R$ is monomorphic. In this case, we   prove that $\sib(R)$ is one, in which case $Aut(R)$ is the full symmetric group,  or $2^{\aleph_0}$ (Theorem \ref 
{thm:sibrational}). To do this we use both Frasnay's result on chainable structures and  Cameron's result  on monomorphic groups. More generally, we show that if  $R$ has an infinite class which is not a strongly  indiscernible subset of $R$ (that is some permutation of that class does not extend to an automorphism of $R$ by the identity on the remainder) then $R$ has $2^{\aleph_0}$ siblings  (Theorem \ref{prop:keyfinite}). From this, it follows that if $R$ has a finite monomorphic decomposition then $R$ has one or $2^{\aleph_0}$ siblings (Theorem \ref{thm:monomorphic-siblings}). Next, we consider the case where  
$R$ has no finite monomorphic decomposition. Here,  we prove that $\sib(R)$ is infinite ($(a)$ of Theorem \ref{thm infinitely many}). Indeed, since $R$  is universal for its age, every countable extension with the same age will be equimorphic to $R$. With Ramsey's theorem and the compactness theorem of first order logic, we can build an extension $R'$ of $R$  whose domain $E'$ is an extension of the domain $E$ of $R$, and where $E'\setminus E$ is an infinite  monomorphic part of $R'$. 
Then for $H$  a finite  subset of $E'\setminus E$ and $R'_H=R'{\restriction E\cup H}$, we will obtain that $R'_H$ is equimorphic to $R$. Our aim is to get $R'$ such that for infinitely many integers $k$, the various $R'_H$'s with $\vert H\vert=k$ are pairwise non isomorphic, hence  $\sib(R)$ will be infinite. Using  the fact that $R$ has infinitely many components, we get $R'$ such that the trace over $E$ of the component of $R'$ containing $E'\setminus E$ is finite. This will suffices to realize our aim. Finally, using again the compactness theorem of first order  logic, we prove that, if $R$ has infinitely many infinite components, then $\sib(R)= 2^{\aleph_0}$ siblings ($(b)$ of Theorem \ref{thm infinitely many}). 

The value of $\sib(R)$ remains unsettled if $R$ has infinitely many finite monomorphic classes and all infinite classes are strongly indiscernible. If $R$ is the Random graph or an infinite direct sum of copies of the complete graph $K_m$ ($m\in \N$) all classes are finite (the classes of the Random graph are singletons, while the class of the direct sum of copies of $K_m$ are these copies). But, the number of siblings of the Random graph is the continuum, while  the number of siblings of this direct sum is countable. We conjecture that $\sib(R)$ is at most countable if and only if $R$ is cellular. (see Problem \label{prob:improvement}\ref{section:thelast} in Section \ref{section:thelast}).  

%
\subsection{Structure of the paper} Basic definitions are introduced in Section 2.  
Five sections focus on the proof of the main theorem. In Section 3 we present the notion of monomorphy and prove  that  if  a countable relational  structure  $R$ is monomorphic, uniformly prehomogeneous and  if $Aut(R)$ is not the symmetric group then $\sib(R)=2^{\aleph_0}$. (Theorem \ref{thm:sibrational}). We introduce in Section \ref{monomorphic decomposition} the notion of  monomorphic decomposition of a relational structure.  In Section \ref{Finite monomorphic decomposition} we prove that if a countable relational structure is uniformly prehomogeneous and has  a finite monomorphic decomposition then it has $1$ or $2^{\aleph_0}$ siblings(Theorem \ref{thm:monomorphic-siblings}). In Section \ref{section:nofinitemonomorphicdecomposition} we consider the case of structures without finite monomorphic decomposition. We reassemble our results  in Theorem \ref{thm:main2} of  Section \ref{section:maintheorem}. Theorem \ref{thm:main} follows. 
 
 In Section \ref{section:thelast}, the last section, we present several problems around the notion of equimorphy.

\section*{acknowledgement} We thank the organizers of the Banff International Research Station workshop on Homogeneous Structures (15w5100), held November 8-13 2015, where a preliminary version  of this paper was presented.

We are pleased to thank the referee  of this paper for thoughtful suggestions  and  corrections.

\section{Basic definitions}\label{section: basic}
 
Our terminology follows that of Fra\"{\i}ss\'{e} \cite{fraisse}. A \textit{relational structure} of \emph{signature} $\mu=(n_i)_{i\in I}$ and \emph{domain} $E$ is a pair $R= (E,(\rho_i)_{i\in I})$ where each $\rho_i$ is an $n_i$-ary relation on $E$.  If $I'$ is a subset of $I$, then $R'= (E,(\rho_i)_{i\in I'})$ is  called a \emph{reduct} of $R$, and called a  \emph{finite reduct} if $I'$ is finite. A relational structure $R= (E,(\rho_i)_{i\in I})$ is a  \emph{binary relational structure}, \emph{binary structure} for short, if it is made only of binary relations. It is \emph{ordered} if one of the relations $\rho_i$ is a linear order. 

\subsection{Embeddability, age, profile}
The \emph{substructure induced by $R$ on a  subset $A$} of $E$, simply called the \emph{restriction of $R$ to $A$}, is the relational structure $R_{\restriction A}= (A,
(A^{n_i}\cap \rho_i)_{i\in I})$. For simplicity the restriction to $E\setminus \{x\}$ is denoted $R_{-x}$.  The notion  of \emph{isomorphism} between relational structures is defined in the natural way. A map $f$ from a subset $F$ of the domain $E$ onto a subset $F'$ of a relational structure $R'$ is a \emph{local isomorphism of $R$ into $R'$} if $f$ is an isomorphism of $R_{\restriction F}$ onto $R'_{\restriction F'}$. If $R=R'$, we say that $f$ is  a \emph{local isomorphism of $R$} (or a local embedding of $R$). A relational structure $R$ is \emph{embeddable} into a relational structure $R'$ if $R$ is isomorphic to some restriction of $R'$. Embeddability   is a quasi-order on the class of structures having a given signature.

The \emph{age} of a relational  structure $R$ is the set $age(R)$ of restrictions of $R$ to finite subsets of its domain, these restrictions being considered up to isomorphy. The \emph{profile} of a relational structure $R$ is the function $\varphi_R$ which gives for every non-negative integer $n$, the number of $n$-element restrictions counted up-to isomorphy. This function depends only on the age of $R$.

\subsection{Homogeneity}

A relational structure $R$ is  \emph{homogeneous} if  every finite local isomorphism extends to an automorphism of the structure (the notion has been introduced independently by several authors, the current terminology comes from Fra\"{\i}ss\'e; the reader must be aware that it is called ultra-homogeneous in some of the early literature).  We present below three generalizations of this notion. We focus on the  notion of \emph{uniform prehomogeneity} which we characterize in term of the notion of local $1$-embedding. 

Let $R$ and $R'$ be two relational structures on $E$ and $E'$ respectively; we say that a map $f$ defined on a subset $F$  of $E$ with values in a subset $F'$ of $E'$ is a \emph{local $1$-embedding of $R$ into $R'$} if its restriction to every finite subset $H$ of $F$  extends  to every finite set $\overline {H}\subseteq E$ containing $H$ to a local  isomorphism of $R$ into $R'$. If $f^{-1}$, the set inverse of $f$, is also a local $1$-embedding, we say that $f$ is a \emph{local $1$-isomorphism}; if such $f$ exists, we say that  $F$ and $F'$  are  \emph{$1$-isomorphic} or have the same \emph{$1$-isomorphism type}.

Let $R$ be a relational structure with  base  $E$. An \emph{extension} of $R$ is any  relational structure $R'$  such that $R'_{\restriction E}= R$. An extension $R'$  is a  $1$-\emph{extension  of  $R$} if  for every finite subset  $F$ of $E$, the identity map $\Id_{\restriction F}$ on $F$ is a $1$-local embedding \emph{from $R'$ to $R$}. This means that for 
every finite subset  $F'$ of  $E'\setminus E$ there is a local isomorphism of $R'$ to $R$ which is the identity on $F$  and maps  $F'$ into $E$. Then, we say that a  relational structure $R$ is \emph{existentially closed} if every  extension of $R$ with the same age is a $1$-extension. We say that $R$ is \emph{existentially universal} if for every extension $R'$ with the same age,  every finite $F$ in the domain of $R$, every finite $F'$ in the domain of $R'$, the identity map on $F$ extends to $F'$ to a local $1$-embedding of $R'$ to $R$ (its role is discussed in the last Section).

We say that  $R$ is \emph{prehomogeneous} if, for every finite set $F$ of the domain $E$ of $R$, there is a finite superset $F'$ of $F$ such that every local isomorphism of $R$ with domain $F$ extends to an automorphism of $R$ provided that it extends to $F'$. We say that $R$ is  \emph{uniformly prehomogeneous} if in addition the  cardinality of $F'$ is bounded by some function $\theta$ of the cardinality of $F$.

Slightly different notions of existentially closed  and existentially universal structures were introduced by Robinson in syntactical terms by means of existential sentences and existential types \cite{robinson}. Notions of prehomogeneity and uniform prehomogeneity were introduced by Pabion \cite{pabion} for multirelations (relational structures with finitely many relations); a syntactical definition is in \cite{pouzet}. 
If the profile of $R$ takes only integer values (particularly if the signature is finite), our definitions given here are  equivalent to the syntactical definitions. In this case, $(a)$ every structure extends to an existentially closed structure with the same age; $(b)$ $R$ is existentially closed iff every  local $1$-embedding of $R$ with finite domain in an extension with the same age is a local $1$-isomorphism.

A characterization of  prehomogeneity was given by Pabion (Proposition 1, p. 530, \cite{pabion}) for multirelations. It is given in terms of complete types. With our condition below,  his proof extends to structures with infinitely many relations. For more about prehomogeneity, see \cite{pouzet2, simmons, pouzet-roux}. 

\begin{theorem}\label{thm:prehomogeneity}
A relational structure $R$ on a countable  set $E$  is prehomogeneous if and only if  for each finite subset $F$ of $E$ there exist $\overline F$ finite containing $F$ such that every local isomorphism defined on $F$ which extends to $\overline F$ is a local $1$-isomorphism. 
\end{theorem}

The following result summarizes the main properties of uniform prehomogeneity. Equivalences from $(ii)$ to $(v)$ are in Proposition 3, p.531 of \cite {pabion}, (see also Proposition 3.1, p.696 of \cite{pouzet}); statement $(i)$ is new.

\begin{theorem} \label{thm:unifprehomtest} 
Let $G$ be a permutation group acting on a countable set $E$, and $R$ be a relational structure  on $E$. Then the following properties are equivalent: 
\begin{enumerate}[{(i)}]
\item $G$ is oligomorphic, $Aut (R)= \overline G^{\mathfrak S}$ and $Emb (R)$, the monoid of embeddings of $R$, is equal to $\overline G$; 

\item $R$ is uniformly prehomogeneous and its profile takes only finite values; 

\item  $(a)$ every local $1$-embedding of $R$ with finite domain is a local $1$-isomorphism and   $(b)$ for each integer $n$, the number of $1$-isomorphism  types of $n$-element subsets of $R$ is finite; 
\item $R$ is prehomogeneous and $Aut(R)$ is   oligomorphic; 
\item $R$ is $\aleph_{0}$ categorical and $Th(R)$ is axiomatizable by universal-existential sentences. 
\end{enumerate}

 \end{theorem}
 
\begin{proof}
Given Pabion's result (Proposition 3, p.531 of \cite {pabion}) it is actually enough to show $(i)\Rightarrow (iii)$ and $(iv)\Rightarrow (i)$. However, we also show that $(iii)\Rightarrow (i)$. 

$(i)\Rightarrow (iii)$. Since $G$ is oligomorphic, $Aut(R)$ is oligomorphic too, hence $(b)$ holds. To prove that $(a)$ holds, let $f$ be a local $1$-embedding of $R$  mapping a finite subset $F$ of $E$ onto $F'$. The map $f$ extends to an embedding $\overline f$ from $R$ into some extension $R'$, such that $F'$ has the same $1$-isomorphism type in $R$ and in $R'$ (indeed, if $\vert F\vert= n$,  add $n$ constants to the language of $R$ interpreted as the $n$ elements $a_1, \dots, a_n$ of $F$ and $f(a_1), \dots, f( a_n)$ of  $F'$;  the universal theory $T$ of $(R, a_1, \dots, a_n)$  contains the universal theory $T'$ of $(R, f(a_1), \dots,f(a_n))$  hence there is  some  extension  $R'$ of $(R, f(a_1), \dots,f(a_n))$ whose universal theory is $T'$ (a well known consequence of  the Compactness theorem of first order logic). Since $Aut(R)$ is oligomorphic, the profile of $R$ takes only finite values, hence the identity map on  $F'$ is a local $1$-isomorphism of $R$ into $R'$. Now, since 
 $Aut(R)$ is oligomorphic, the complete theory of $R$ is $\aleph_0$-categorical. It follows  that $R$ is universal in the universal theory of $R$ and thus there is an embedding 
 $g$ of $R'$ into $R$. The map $g\circ \overline f$ is an embedding of $R$, hence, according to our hypothesis, its   restriction  to $F$ is the restriction of an automorphism. It follows that this restriction is a $1$-embedding, hence $f$ is a $1$-isomorphism, as claimed.  
 
 $(iii)\Rightarrow (iv)$. The condition in Theorem \ref{thm:prehomogeneity} is satisfied, hence $R$ is prehomogeneous. Due to $(iii)(b)$, $Aut(R)$ is oligomorphic.  

$(iv)\Rightarrow (i)$.  It suffices to see that $Emb(R)= \overline{Aut(R)}$. Without any condition, $\overline{Aut(R)}\subseteq Emb(R)$. Let $g\in Emb(R)$. We need to show that given $F$, finite subset of $E$, there is an automorphism $\overline g$ which agrees with $g$ on $f$. Since the  restriction $g_{\restriction F}$ of $g$ to $F$ extends to every finite subset of $E$ it extends to $\overline F$; since $R$ is prehomogeneous,  $g_{\restriction F}$ extends to an automorphism $\overline g$, hence $g\in \overline{Aut(R)}$. 
 

\end{proof}

 Since our main result is on $\aleph_0$-categorical structures which have finite profile,  \textbf{we consider only relational structures with finite profile.} This allows us to code restrictions of such relational structures by open formulas.

\section{The number of siblings of monomorphic  structures}\label{monomorphic structure}

The purpose of this section is to prove a first result which allows us to count the number of siblings based on structural properties.

\begin{theorem}\label{thm:sibrational}
If a countable relational structure $R$ is monomorphic,  uniformly prehomogeneous and $Aut(R)$ is not the symmetric group, then $sib(R)= 2^{\aleph_{0}}$. 
\end{theorem} 

\subsection{Free-interpretability, chainability and monomorphy} 

Let $R$ and $S$ be two  relational structures on the same domain $E$. We say that $R$ is \emph{freely interpretable by}   $S$ if every local isomorphism of $S$ is a local isomorphism of $R$.  If $S$ is a chain, we say that $S$ \emph{chains} $R$, and thus we say that $R$ is \emph{chainable} if some chain $S$ chains $R$.

Now let $p$ be a non-negative integer; a relational structure $R$ is said to be \emph{$p$-monomorphic}  if its restrictions to finite sets of the same cardinality $p$ are all isomorphic; the relational structure is \emph{monomorphic} if it is $p$-monomorphic for every $p$. Since two finite chains with the same cardinality are isomorphic, chains are monomorphic structures and hence so are chainable relational structures. Conversely, Fra\"{\i}ss\'e \cite {fraisse} showed that every infinite monomorphic relational structure is chainable. 


\medskip

We now consider three well known structures associated to a chain $C= (E \leq)$:
 \begin{itemize}
\item The \emph{betweeness relation $B_C= (E, b_C)$ associated to $C$},  where $b_C$ is the set of triples $(x_1,x_2,x_3)$ such that either $x_1<x_2<x_3$ or $x_3<x_2<x_1$.
\item The \emph{circular order  $T_C= (E, t_C)$ associated to $C$}, where $t_C$ is the set of triples 
$(x_1,x_2,x_3)$ such that $x_{\sigma(1)}<x_{\sigma(2)}<x_{\sigma(3)}$ for some circular permutation $\sigma$ of $\{1,2,3\}$.  
\item The \emph{betweeness relation  $D_C=(E, d_C)$ associated to the circular order}, where $d_C$ is the set of quadruples 
$(x_1,x_2,x_3, x_4)$ such that $x_{\sigma(1)}<x_{\sigma(2)}<x_{\sigma(3)}<x_{\sigma(4)}$ or $x_{\sigma(4)}<x_{\sigma(3)}<x_{\sigma(2)}<x_{\sigma(1)}$ for some circular permutation $\sigma$ of $\{1,2,3, 4\}$.  
 \end{itemize}
 
 By construction, these three structures are chainable by $C$. Furthermore, if $C$ is isomorphic to the chain of rational numbers, these three structures are actually homogeneous.

  Moving to group properties and following Cameron \cite{cameron}, a  group of permutations on a set $E$  is \emph{monomorphic} if it has just  one orbit for every $n$-element set (another terminology is \emph{set-homogeneous}). Cameron  proved that on a countable set there are essentially five monomorphic closed groups: 
   \begin{theorem}\cite{cameron}\label{thm:cameronthm}
   A monomorphic closed group on a countable set is isomorphic, as a permutation group, to one of the following groups:
   \begin{enumerate}[{(a)}]
   \item $\mathfrak S(\Q)$, the full symmetric group on the set of  rationals;
   \item $Aut(\Q)$, the automorphism group of the chain of rational numbers; 
   \item $Aut (B_{\Q})$ the automorphism group of the betweeness relation associated to the chain of rational numbers;  
   \item $Aut(T_{\Q})$ the automorphism group of  the circular order associated with the chain of rational numbers; 
  \item   $Aut(D_{\Q})$ the automorphism group of the betweeness relation  associated to the circular order on the rationals.
   \end{enumerate}
\end{theorem} 

We will need the following consequence of Theorem \ref{thm:cameronthm} in the proof of $(a)$ of Lemma \ref{lem:notfull}. 

\begin{lemma}\label{lem:nodescending} A  descending chain of monomorphic closed groups on a countable set has at most  four terms.
\end{lemma}
\begin{proof}It suffices to show that if $R$ and $R'$ are two relational  structures on the same set $E$ such that $Aut(R')\subseteq Aut(R)$ and $R$ isomorphic to $R'$, then $Aut(R)=Aut(R')$. This conclusion in fact simply follows if  $Aut(R)$ and $Aut(R')$ are oligomorphic. Indeed, for each integer $n$, the partition of $E^{n}$ into orbits for the action of $Aut(R')$ is included into the partition of $E^{n}$ into orbits for the action of $Aut(R)$. Since these groups are oligomorphic and isomorphic as permutation groups, these partitions have  finitely many classes  and have the same number of classes, hence are equal. The fact that these groups are equal follows.
\end{proof}

Cameron's theorem implies  that every monomorphic closed group $G$ on a countable set  is the automorphism group of  some relational structure $R$ chainable  by  a chain isomorphic to the chain of rational numbers. In fact, it implies that every $R$ such that $Aut(R)= G$ has this property, and thus the following result. 
 
\begin{theorem}\label{freely rational}  Let   $R$ be  a countable structure; then the following properties are equivalent:
\begin{enumerate}[{(i)}]
\item  $R$ is chainable  by a chain  isomorphic to the chain  of rational numbers; 
 \item $R$ is monomorphic and uniformly prehomogeneous;  
 \item $Aut(R)$ is monomorphic. 
 \end{enumerate}\end{theorem}

Theorem \ref {freely rational}  was proved in \cite {pouzet 81} ( see 2.6 and 2.7  and line 18 of page 321) by direct arguments. It was a step in a proof of Cameron's theorem based on Frasnay's result. We outline a proof.

 \begin{proof} $(i)\Rightarrow (iii)$.  If $R$ is chainable  by the chain of rational numbers then $Aut(R)$ is an overgroup of $Aut(\Q)$ hence it is monomorphic. \\
$(iii)\Rightarrow (ii)$. If $Aut(R)$ is monomorphic then trivially, $R$ is not only monomorphic but any two finite subsets of the same size are $1$-isomorphic.  
Thus to show that $R$ is uniformly prehomogeneous, it suffices by Theorem \ref{thm:unifprehomtest} to show that every local $1$-embedding $f$ with finite domain $F$ included in the domain $E$ of $R$ is invertible by a  local  $1$-embedding. Indeed, let $F'$ be the image of $F$. Since $Aut(R)$ is monomorphic there is an automorphism, say $\sigma$, which carries $F'$ onto $F$. Evidently, $\sigma$ is a $1$-local embedding, hence $\sigma_{\restriction F'} \circ f$ is a $1$-local  embedding;   furthermore all the iterates of that map are $1$-local embeddings. Since $F$ is finite, an $n$-th iterate is the identity on $F$, hence $f^{-1}= (\sigma \circ f)^{n-1}\circ \sigma $ is a $1$-local embedding as claimed.  \\
 $(ii)\Rightarrow (i)$. Part of the argument  is based on the following fact about free interpretability. Let $R$ and $S$ with the same base and let $\mu$ and $\nu$ be the respective signatures of $R$ and $S$. If the signature $\nu$ is finite, then $R$ is freely interpretable by $S$ if and only if there exists a map $P$ associating to every relational structure $S'$ of signature $\nu$ a relational structure $R'$ of signature  $\mu$ on the same domain in such a way that $(a)$ $P(S)=R$ and $(b)$ every local isomorphism $f$ of $S'$ into $S''$ is a local isomorphism of $P(S')$ into $P(S'')$ (see Fra\"{\i}ss\'{e} \cite{fraisse}). Now the proof of the implication goes as follows.  Suppose that  $R$  is monomorphic, then $R$ is chainable  by some chain, say $C$. The free operator transforming $C$ into $R$ will transform $\Q$ into some structure $R'$.  It turns out that $R'$ is isomorphic to $R$. Indeed,  according to the implication $(i)\Rightarrow (ii)$, already proven, $R'$ is uniformly prehomogeneous;  it has the same age as $R$ which is uniformly prehomogeneous, hence it is isomorphic to $R$ (this is essentially the argument in \cite {pouzet 81}, 2.5 Lemme de pr\'eservation, page 320). Hence $R$ is  chainable  by some chain $D$ isomorphic to the chain of rationals.   \end{proof} 

 \begin{remark}
If $R$ is only known to be chainable, then it does not follow that $Aut(R)$ is monomorphic, even if $Aut(R)$ is oligomorphic as the example  $R=1+\Q$ shows.
\end{remark}

\subsection{Group-sequences, bichains and indicative sequences}

\subsubsection{Group-sequences}
Let $R$ be a  chainable relational structure with domain $E$,  and $C$ be a chain (with same domain $E$) chaining $R$. Let $n$ be an integer, $n\leq \vert E\vert$, let $A$ be a $n$-element subset of  $E$ and $c_ A$ be the unique isomorphism of the natural chain on $\underline{n}=\{1, \dots,  n \}$ onto $C_{\restriction A}$.  The set of permutations $\sigma$ of $\underline n$ of the form $c_A^{-1}\circ \tau \circ c_A$ for $\tau \in Aut(R_{\restriction A})$ forms  a group. Since $C$ chains $R$, this group is independent of the $n$-element set $A$ and we denote it by $Ind_n (R,C)$. The sequence of these groups is called the \emph{group-sequence} of the pair $(R,C)$. 
 
For each positive integer $n$ we define the following permutation groups on $\underline n$:
\begin{itemize}
\renewcommand{\labelitemi}{$\bullet$}
\item $\mathfrak S(n)$ consisting of  all permutations;
\item $\mathfrak I(n)$ consisting of only the  identity; 
\item  $\mathfrak J(n)$  consisting of  the identity and the reversal $r$ transforming each $k$ into $n-k+1$; 
\item $\mathfrak T(n)$ consisting  of circular permutations;
\item $\mathfrak D(n)$  consisting of the product of  $\mathfrak T(n)$ and $\mathfrak J(n)$, that is the dihedral group. 
\end{itemize}

Let $\mathfrak S$, $\mathfrak I$, $\mathfrak J$, $\mathfrak T$, $\mathfrak D$ each be the sequence of the above corresponding groups for $n\in \N$.
Then clearly we have the following result connecting these sequences and our previous structures. 

\begin{lemma}\label{lem:groupseq} The sequences   $\mathfrak S$, $\mathfrak I$, $\mathfrak J$, $\mathfrak T$, $\mathfrak D$ are the sequences $(Ind_n(R, \Q))_{n\in \N}$ where $R$ is successively $(\Q, =)$, $(\Q, \leq)$, $B_{\Q}$,  $T_{\Q}$ and  $D_{\Q}$. 
\end{lemma}

\subsubsection{Bichains and their indicative sequences}
As before let $R$ be a  chainable relational structure with domain $E$,  and $C=(E,\leq)$ be a chain chaining $R$. We may observe that for every  embedding $\varphi$ of $R$ into $R$,   the inverse image of $\leq$ by $\varphi$ again provides a chain chaining $R$.  Frasnay  \cite{frasnay 65}  studied the relationship between two chains  chaining the same structure, and we briefly  recall some elements of his theory (for more,  see \cite {frasnay 65}, \cite{frasnay 84}, and Fra\"{\i}ss\'e \cite{fraisse}). 
  
A \emph{bichain} is a relational structure with two linear orders on the same set. To each bichain we associate a sequence of permutations groups, called the \emph{indicative sequence} of the bichain.  Consider a  bichain $B=(E, \leq_0, \leq_1)$, and set each component as  $B_i= (E, \leq_i)$ for $i=0, 1$. Let $n$ be a positive integer no larger than the cardinality of $E$, and let $A$ be an $n$-element subset of $E$. The chains $B_{0}{\restriction A}$ and $B_{1}{\restriction A}$ are isomorphic via  a unique permutation $h$ of $A$ which transforms the first to the second; if we order $A$ into the sequence $a_1<_0 \dots <_0a_n$, there is a unique permutation $\sigma$ of $\underline {n}= \{1, \dots n\}$ which reorders it into $a_{\sigma(1)}<_1 \dots <_1 a_{\sigma(n)}$,  that is satisfies $h(a_k)= a_{\sigma (k)}$ for $k\in \{1, \dots, n\}$. The collection of these permutations $\sigma$ for $n$ fixed, $A$ belonging to all the $n$-element subsets of $E$,  generates a subgroup $Ind_n(B)$ of $\mathfrak S(n)$, called the \emph{$n^{th}$ indicative group of $B$}. The sequence of these indicative groups is the \emph{indicative sequence} of $B$. We can now recall the following result of Frasnay (\cite {frasnay 84}  Lemme, page 263).

\begin{theorem}\label{thm:frasnay84} \cite {frasnay 84}  Let $B=(E, \leq_0, \leq_1)$ be a bichain. If $B_0$ has no minimum and no maximum then the indicative sequence of $B$ is one of the five sequences $\mathfrak S$, $\mathfrak I$, $\mathfrak J$, $\mathfrak T$, $\mathfrak D$ listed above. 
\end{theorem}

This together with   Lemma \ref{lem:groupseq} yields the following.

\begin{corollary}\label{cor:frasnay}
The indicative sequence of a bichain whose components have no extreme elements is the group-sequence of a homogeneous monomorphic countable structure. 
\end{corollary} 

Recall that a chain is \emph{scattered} if it does not embed the chain of the rationals.

\begin{lemma}\label{lem:scatteredbichain}  Let $B=(E, \leq_0, \leq_1)$ be a bichain such that $B_0$ is non-scattered and $B_1$ is scattered.
Then the indicative sequence of $B$ is $\mathfrak S$. 
\end{lemma}

\begin{proof} We will make use of the following.

\begin{claim}\label{claim:rationals}
There is a subset $A$ of $E$ such that $B_0 \restriction A$ is isomorphic to the chain of rational numbers, and $B_1 \restriction A$ is isomorphic to either $\omega$ or $\omega^*$.
\end{claim}

\noindent{\bf Proof of Claim \ref{claim:rationals}.}
This readily follows from a famous unpublished result of Galvin, 
expressing that if the pairs of rational numbers are divided into finitely many classes, then  there is a subset of the rationals which is isomorphic to the rationals and such that all pairs are contained in the union of at most two classes; for a proof see Todorcevic \cite{todorcevic} (Theorem 6.3  Page 44), or Vuksanovic \cite{vuksanovic}.

Indeed, pick a subset $E'$ of $E$ such that $B_0 \restriction E'$ is isomorphic to the rationals, and let $\leq_2$ be an ordering of $E'$ in type $\omega$. Then distribute the pairs $(x,y)$  of $E'$ with $x <_0y$ into four classes according to how $x$ and $y$ compare with $\leq_1$ and $\leq_2$. Galvin's theorem yields a subset $A$ of $E'$ such that $B_0 \restriction A$  is isomorphic to the rationals, and $B_1 \restriction A$ either agrees with $\leq_2$ or its reverse, hence either of type $\omega$ or $\omega^*$.

\hfill $\Box$ 

We can thus assume that  $B=(E, \leq_0, \leq_1)$ is a bichain where $B_0$ is isomorphic to the rationals and $B_1$ is  isomorphic to $\omega$.  Under this assumption we have the following.

\begin{claim}\label{claim:allpermutations} 
For each integer $n$, the set of permutations $\sigma$ of $\underline{n}$ which reorders an  $n$-element  ordered set $a_1<_0 a_2<_0 \cdots <_0 a_n$ of $A$ into  $a_{\sigma(1)} <_1 a_{\sigma(2)} <_1 \cdots <_1 a_{\sigma(n)}$ is the full symmetric group ${\mathfrak S}(n)$.
\end{claim}

With these claims it follows that the indicative sequence of (the original) $B$ is $\mathfrak S$ as required. 

\medskip

\noindent{\bf Proof of Claim \ref{claim:allpermutations}}
We proceed by induction on $n$. Let $\sigma \in {\mathfrak S}(n)$  and let $i=\sigma(n)$. It suffices to consider the case $1<i<n$, and by induction we may find an  $n-1$-element ordered set $a_1<_0 a_2 <_0 \cdots <_0 a_{i-1} <_0 a_{i+1} <_0 \cdots <_0  a_{n}$ of $A$ such that $a_{\sigma(1)} <_1 a_{\sigma(2)} <_1 \cdots <_1 a_{\sigma(n-1)}$. Since the interval $(a_{i-1}, a_{i+1})$ in $B_0$ is infinite and there only finitely many elements less than $a_{\sigma(n-1)}$ in $B_1$, we may find $a_i$ such that $a_{i-1}<_0 a_i<_0 a_{i+1}$ and $a_{\sigma(n-1)} <_1 a_i$. Then  $\sigma$ reorders this $n$-element set as required for the claim.
\hfill $\Box$ 

\medskip
This completes the proof of Lemma \ref{lem:scatteredbichain}.
\end{proof}

\medskip

From this, we can deduce the following which is key to our structural result.

\begin{theorem}\label{thm:siblingmonomorphic} Let $G$ be a monomorphic closed group on a countable set. Then all  relational structures $R$  such that $Aut(R)= G$ have the same number of siblings: this number is $1$ if $G$ is the full symmetric group,  and $2^{\aleph_{0}}$ otherwise. 
\end{theorem}

\begin{proof}
  Suppose  first that  $R$ is one of the the five  previously listed homogeneous relational structures defined on $\Q$. If $R$ is the equality relation, there is just one sibling. If $R$ is one of the four others, we prove that there are $2^{\aleph_0}$ non isomorphic siblings. For that, we define subsets  $C_s$ of $\Q$ for each $s\in \{0,1\}^{\N}$ such that the restrictions $R_{\restriction C_s}$ are equimorphic to $R$ and pairwise non isomorphic. The structure of the $C_s$'s is such that an isomorphism of some $R_{\restriction C_s}$ onto some $R_{\restriction C_{s'}}$ will necessarily be an isomorphism from the chain $C_s$ onto the chain $C_s'$ and hence $s=s'$.  Each $C_s$ is the union of  three sets  $A_0$, $A_1^s$ and $A_2$, where   $A_0$ and $A_2$ are respectively a non-empty  initial and final segment of $\Q$ without a largest element and a least element, and $A_1^s$ is a scattered chain of the form $\sum_{n<\omega} C_n^{s(n)}$, where $C_n^{s(n)}$ is a chain of order type $\omega$, resp. $\omega^{*}$,  if $s(n)=0$, resp. $s(n)=1$. It can be easily verified  that distinct sequences provide non-isomorphic chains.  But now if $R$ is any of the other four homogeneous structures, then each structure $R{\restriction C_s}$ is a sibling of $R$, and   an isomorphism from $R{\restriction C_s}$ onto $R{\restriction C_{s'}}$  has to be an order isomorphism from $A_1^s$ onto $A_1^{s'}$ or its reverse; the first case happens only if $s=s'$ while  the second case never happens, due to the form of $A_1^s$ and $A_1^{s'}$. Hence $\sib(R)= 2^{\aleph_0}$ as required.
   
\medskip

Next we deal with the general case. According to Theorem \ref{freely rational} we may suppose that $R$ is chainable by the chain $C= (\Q, \leq)$ of rational numbers and furthermore that $G= Aut (M)$ for some of the homogeneous relations  occuring in Cameron's theorem; we show that $\sib(R)=\sib(M)=2^{\aleph_0}$.  If  the restrictions of $M$ to two subsets $A$ and $A'$  
 are isomorphic,  then the restrictions $R_{\restriction A}$ and $R_{\restriction A'}$ are isomorphic. From this, and the fact that the embeddings of $R$  coincide with the embeddings of $M$,  it  follows that $\sib(R) \leq \sib (M)$.  

Conversely, suppose that $R_{\restriction A}$ and  $R_{\restriction A'}$ are isomorphic. It suffices to prove the following.

 \noindent \begin{claim}\label{claim:iso} Every  isomorphism $f$ of $R_{\restriction A}$ onto $R_{\restriction A'}$ is a local isomorphism of $M$ provided that $A$, as a subset of $\Q$,  has no extreme elements. 
\end{claim}

\medskip

Indeed to conclude the proof using the claim, if we take $2^{\aleph_0}$ subsets $C_s$ of $\Q$ with no extreme elements such that  their restrictions to $M$ are pairwise non isomorphic and equimorphic to $M$, as we did above, then by the claim the restrictions of  $R$ to these will also yield pairwise non isomorphic and equimorphic structures to $R$,  and hence $2^{\aleph_0}=\sib(M)\leq\sib(R)$.

\medskip

Now toward proving the claim, consider the $n$-th indicative group $Ind_n(B)$ associated to the bichain $B=(A,\leq_A,\leq_A')$, where $\leq_A'$ is the  image of $\leq_{A'}$ by $f^{-1}$, and also consider the group-sequences $Ind_n(M, \Q)$ and  $Ind_n (R, \Q)$. In order to prove the claim it suffices to prove the following:

\noindent \begin{subclaim}  
$Ind_n(B) \subseteq  Ind_n(M, \Q)$ for each integer $n$. 
\end{subclaim}

Indeed, let $A_n$ be an $n$-element subset of $A$, let $\sigma$ be the permutation of $\{1, \dots n\}$ such that if $a_1<_A\dots <_Aa_n$ is an enumeration of $A_n$, then the sequence $a'_1<_{A'}\dots <_{A'}a'_n$ with $a'_i= f(a_{\sigma(i)})$ provides an enumeration of $A'_n=f(A_n)$. Then by definition $\sigma\in Ind_n(B)$, and thus  if the subclaim holds we have $\sigma\in   Ind_n(M, \Q)$, and hence $\sigma^{-1}$  as well.  Now let $t$ be the unique order-isomorphism from $f(A_n)$ onto $A_n$  and define $g=t\circ f_{\restriction A_n}$; this map is represented  on $\{1, \dots n\}$ by  $\sigma^{-1}$ hence it is an automorphism of $M_{\restriction A_n}$. It follows that $f$ induces an isomorphism from $M_{\restriction A_n}$ onto $M_{\restriction A'_n}$ from which follows that $f$ is a local isomorphism of $M$. 

\medskip

\noindent {\bf Proof of the  subclaim.}  

We have easily $Ind_n(B)\subseteq Ind_n(R, \Q)$.  According to Frasnay's Theorem \ref{thm:frasnay84} above, $(Ind_n(B))_n$ is the group-sequence of some homogeneous structure  belonging to the Cameron list, and thus let $M'$ be such a structure with domain $\Q$. We have $Ind_n(M', \Q))\subseteq Ind_n(R,\ Q)$, which implies $ Aut (M') \subseteq Aut (R)$. But now since $Aut(R)= Aut(M)= G$  the subclaim follows. 

\medskip

This completes the proof of Theorem \ref{thm:siblingmonomorphic}.
\end{proof}

\medskip

\subsubsection{Proof of Theorem \ref{thm:sibrational}}
Let $R$ be monomorphic and uniformly prehomogeneous  such that  $Aut(R)$ is not the symmetric group. Then, according to Theorem \ref{freely rational},    $Aut (R) $ is monomorphic, and thus $\sib(R)= 2^{\aleph_{0}}$  by Theorem \ref{thm:siblingmonomorphic}.

\section{Monomorphic decomposition of a relational structure}\label{monomorphic decomposition}

In this section, we extend some notions of the previous section bringing the  concept of monomorphic decomposition of a relational structure into play, and we specialize it to  permutation groups. This notion was introduced in \cite{pouzet-thiery} and will form a main tool in this work. Our presentation  follows \cite{oudrar-pouzet}, see Chapter 7 of \cite{oudrar} for details.                                     

Let $R$ be a relational structure on a set $E$. A subset $E'$ of $E$  is a \emph{monomorphic part} of $R$ (or a monomorphic block) if for every integer $k$ and every pair $A,~A'$ of $k$-element subsets of $E$, the induced structures on $A$ and $A'$ are isomorphic whenever $A\setminus E'=A'\setminus E'$ (we do not require that an isomorphism of $A$ onto $A'$ sends $A\setminus E'$ onto $A'\setminus E'$). A \emph{monomorphic decomposition} of $R$ is a partition of $E$ into monomorphic parts. A monomorphic part which is maximal for inclusion is called a \emph{monomorphic component} of $R$, and together form a  monomorphic decomposition of $ R$ of which  every monomorphic decomposition of $R$ is a refinement (Proposition 2.12 of  \cite{pouzet-thiery}).

This partition can also be defined  in a direct way as follows,  see \cite{sikaddour-pouzet}. For $x$ and  $y$  two elements of $E$ and  $F$  a finite subset of $E\setminus \{x, y\}$, we say that $x$ and  $y$  are \emph{$F$-equivalent}, written  $x \simeq_{F, R} y$,  if   the restrictions of $R$ to   $\{x\}\cup F$  and $\{y\}\cup F$ are isomorphic (we do not require that an isomorphism of $\{x\}\cup F$ onto $\{y\}\cup F$ sends $x$ to $y$). For  $k$ a non-negative integer, we set $x \simeq_{k, R} y$ if $x \simeq_{F, R} y$ for every $k$-element  subset $F$ of   $E\setminus \{x,y\}$.  
We set $x\simeq_{\leq k, R}y$ if $x\simeq_{k', R}y$ for every $k'\leq k$ and $x\simeq_{R}y$ if $x \simeq_{F, R} y$ for every finite set $F$.  The  following property holds (\cite{oudrar-pouzet}; for a proof, see Lemma 7.48 and Lemma 7.49 in Section  7.2.5 of \cite{oudrar}). 
\begin{lemma} \label{lem:equiv} The relations $\simeq_{k, R}$, $\simeq_{\leq k, R}$  and $\simeq_R$ are equivalence relations on $E$. Furthermore, the equivalence classes of $\simeq_R$ are the components  of $R$. 
\end{lemma}

From the definition of these equivalence,  we deduce:
\begin{lemma} \label{lem:restrictionequiv} Let $R$ be a relational structure with base $E$ and $R'$ be the restriction of $R$ to  a subset $E'$  of $E$. If $\vert E' \cap C\vert \geq \Min \{k+2, \vert C\vert \}$ for every equivalence class $C$ of $\simeq_{\leq k, R}$ then  $\simeq_{\leq k, R'}$ coincide with the restriction of  $\simeq_{\leq k, R}$ to $E'$.
\end{lemma}
\begin{proof}Clearly, the  restriction to $E'$ of $\simeq_{\leq k, R}$  is included into  $\simeq_{\leq k, R'}$. For the converse, let $x,y \in  E'$ such that $x \simeq_{\leq k, R'}y$. This means that  for every subset $F'$ of $E'\setminus \{x,y\}$ with at most $k$ elements  the restrictions of $R$ to $\{x\}\cup F'$  and $\{y\}\cup F'$ are isomorphic. Let $F$ be a subset  of $E\setminus \{x,y\}$ with at most $k$ elements.  Since $E'$ keeps at least $k+2$ element of  each equivalence class of  $\simeq_{\leq k, R}$, we may find a subset $F'$ of $E'\setminus \{x, y\}$ such that $\vert F'\cap C\vert = \vert F\cap C\vert$ for every equivalence class $C$ of $\simeq_{\leq k, R}$. As in the proof of Lemma 2.10 of \cite{pouzet-thiery} we may transform $\{x\} \cup F$ into $\{x\} \cup F'$ by adding and removing one element at a time,  from which follows  that the restrictions of $R$ to $\{x\}\cup F'$ and to $\{x\}\cup F$ are isomorphic.  Similarly, the restriction of $R$  to $\{y\}\cup F'$ and to $\{y\}\cup F$ are isomorphic. It follows that the restrictions of $R$ to $\{x\}\cup F$ and to $\{y\}\cup F$  are isomorphic. Hence, $x \simeq_{\leq k, R} y$ as required.
\end{proof}

\begin{lemma} \label{lem:definability} Let $R$ be a relational structure with base $E$, then  there is an integer $k$ such that  the equivalence relations  $\simeq_{\leq k, R}$ and $\simeq_{R}$ coincide, whenever 
\begin{enumerate}
\item $\simeq_{R}$ has finitely many classes or 
\item $Aut(R)$ has finitely many orbits of pairs. 
\end{enumerate}
\end{lemma}

\begin{proof}
Item (1). Let $\ell$ be the number of equivalence classes of $\simeq_{R}$. Pick an element $x_i$ in each  class $X_i$, $i<\ell$. For $i\not =j$ there is a finite set $F_{i,j} \subseteq E\setminus \{x_i,x_j\}$ such that the restrictions of $R$ to $\{x_i\} \cup F_{i,j}$ and $\{x_j\} \cup F_{i,j}$ are not isomorphic. Set $k:= \Max\{\vert F_{i,j}\vert: i,j<\ell\}+1$. 
Item (2). For  each orbit $C$ of a pair $\{x, y\}$ such that $x \not \simeq_{R}y$, witness this fact by selecting a finite subset $F_C\subseteq E\setminus \{x,y\}$. Let $k$ be the maximality of $\vert F_C\vert +1$ where $C$ runs trough these orbits. \end{proof}

A consequence of Item (1) of Lemma \ref{lem:definability} is the following result  (Lemma 2.15 of \cite{pouzet-thiery})  obtained by a more complicated argument.

 \begin{lemma}\label{lem:induced decomposition}If a relational structure $R$ on a set  $E$ has a finite monomorphic decomposition, then there is an integer $d$ such that every finite subset $F$ is contained  in a finite subset $F'$, with $\vert F'\setminus F\vert \leq d$,  and such that the monomorphic decomposition of $R_{\restriction F'}$ into components is induced by the decomposition of $R$ into components. \end{lemma}
%
%

Note that in the case of binary structures or of ordered structures there is a threshold phenomenon indicated below. But, using a result of \cite{pouzet79} one can show that there is no threshold for ternary relations.
\begin{lemma}
The equivalences relations $\simeq_{\leq 6, R}$ and  $\simeq_{R}$ coincide on a binary structure. If $R$ is a directed graph, resp.  an ordered graph, we may replace $6$ by $3$,  resp. by $2$. If $T$ is a tournament,  the number of equivalences classes of  
$\simeq_{\leq 3, T}$ is finite provided that the number of equivalence classes of $\simeq_{\leq 2, T}$ is finite. 
There is an integer $i(m)$ such that on  an ordered structure of arity 	at most  $m$ the equivalences relations  $\simeq_{\leq i(m), R}$ and  $\simeq_{R}$ coincide. 
\end{lemma}

The case of binary structures follows from a reconstruction result of Lopez \cite{lopez 72, lopez 78}. The case of directed graphs was obtained  by Oudrar, Pouzet \cite{oudrar-pouzet} and independently Boudabbous \cite{boudabbous}. The case of ordered structures follows from a result of Ille \cite{ille 92}. 

This notion of equivalence is particularly well adapted for permutation groups. Let $G$ be a permutation group  acting on a set $E$, and $x$ and  $y$ be two elements of $E$. Set $x \simeq_{G} y$ if for every finite subset $F$ of $E\setminus \{x, y\}$, the sets $\{x\}\cup F$  and $\{y\}\cup F$ are in the same $G$-orbit. Now if $R$ is  a  homogeneous relational structure on $E$ such that  $Aut (R)= \overline {G}^{\mathfrak S}$, then clearly $x\simeq_G y$ if and only if $x\simeq_R y$. From this simple observation follows that the relation  $\simeq_G$ is an equivalence relation. We call the equivalence classes, the \emph{$G$-monomorphic components}. 

An immediate consequence of $(2)$ of Lemma \ref{lem:definability} is this:

\begin{corollary}If the automorphism group of a relational structure $R$ is oligomorphic then for some non-negative integer $k$ the equivalence relations  $\simeq_{\leq k, R}$ and $\simeq_{R}$ coincide, hence $\simeq_{R}$ is definable by a universal formula with at most $k$ universal quantifiers. 
\end{corollary}

A crucial use of this notion of  equivalence  is illustrated by the following lemma
\begin{lemma}\label{lem:infinitelymany}
Let $R$ be a relational structure on a set $E$. Suppose that there is some non-negative integer $k$ such that  the equivalence relations  $\simeq_{\leq k, R}$ and $\simeq_{R}$ coincide and furthermore that the size of finite equivalence classes is bounded by some integer $\ell$. If $D$ is an infinite equivalence class, resp., a countable union of infinite equivalence classes then there are $\aleph_0$, resp., $2^{\aleph_0}$,    pairwise non-isomorphic restrictions of $R$ of the form $R_{\restriction E\setminus X}$ where $X$ is a subset of $D$. \end{lemma} 
\begin{proof}
Suppose that $D$ is an equivalence class. Pick in $D$ infinitely many finite subsets $B_n$ with different sizes larger than $\Max \{k, \ell\}+1$. According to Lemma \ref{lem:restrictionequiv},   $B_n$ is a monomorphic component of $R_n:= R_{\restriction (E\setminus D) \cup B_n}$. Since the decompositions of $R_n$ and $R_m$ into monomorphic components do not yield the same sequence of cardinality classes, these structure are not isomorphic. Suppose that $D$   is a countable union of classes, say $C_0, \dots, C_n,  \dots$.  In each $C_n$, pick  a finite set $B_n$ with  size larger than $\Max \{k, \ell\}+1$. Let $s:= \{\vert B_n\vert: n<\omega\}$. According to Lemma \ref{lem:restrictionequiv},   the $B_n$'s  are monomorphic components of $R_s:= R_{\restriction (E\setminus D) \cup \bigcup_{n<\omega} B_n}$. If  $s$ and $s'$ are two different sequences (up to permutations) the monomorphic decompositions of $R_s$ and $R_{s'}$ do not yield the same sequence of cardinality classes hence $R_s$ and $R_s'$ are not isomorphic. Since the number of sequences $s$ as above is $2^{\aleph_0}$, the conclusion follows. 
\end{proof}

We do not claim that the restrictions of $R$ in Lemma \ref{lem:infinitelymany} 
are siblings. We will show in Section \ref{section:nofinitemonomorphicdecomposition} that if $R$ is countable, uniformly prehomogeneous, with infinitely many infinite classes one may  select  a countable union of infinite equivalence classes $C$ such that $R$ is embeddable into $R_{\restriction E\setminus C}$. With this lemma, we get that $R$ has $2^{\aleph_0}$  siblings. 

A variant of these notions is of interest to us. 
A subset $E'$ of $E$  is a \emph{strongly monomorphic part} of $R$ if for every integer $k$ and every pair $A,~A'$ of $k$-element subsets of $E'$ there is an isomorphism of $R_{\restriction A}$ to $R_{\restriction A'}$ which can be extended by the identity on  $E\setminus E'$ to a local isomorphism of $R$.   A \emph{strongly monomorphic component} is a strongly monomorphic part which  is maximal with respect to inclusion (which may not be a monomorphic component).  A \emph{strongly  monomorphic decomposition} of $R$ is a partition of $E$ into strongly monomorphic parts.  
Also, call $E'$ a \emph{chainable part} of $R$ if there is a linear order $\leq$ on $E'$ such that   every local isomorphism of $(E', \leq)$ extended by the identity on $E \setminus E'$ is a local isomorphism of $R$. 

A strengthening of the model theoretic notion of indiscernability  plays a natural role in our context. We say that a subset $E'$ of $E$  is a \emph{strongly indiscernible subset} of $R$ if for every integer $k$ and every pair $A,~A'$ of $k$-element subsets of $E'$ every bijective map from $A$ to $A'$  can be extended by the identity on  $E\setminus E'$ to a local isomorphism of $R$. This amounts to say that every permutation of $E'$ can be extended by the identity on $E\setminus E'$ to an automorphism of $R$.

The following proposition assembles several properties relating these notions. 

 \begin{proposition}\label{prop: strong monomorphic component}
  \begin{enumerate}[{(a)}]
 \item Every strongly monomorphic part is a monomorphic part.
 \item Every strongly monomorphic part is contained in a maximal one, which extends to a monomorphic component. 
\item  There is a strongly  monomorphic decomposition of $R$ from which every other is finer; it is made of  strongly monomorphic components. 
\item A chainable part is a strongly monomorphic part.
\item The converse holds for infinite strongly monomorphic parts, furthermore:  
\item Every infinite monomorphic component is a strongly monomorphic component (and  a chainable part). 
\end{enumerate}
\end{proposition}
The proofs of the first four items are immediate or easy. 
The proof of item (e) uses compactness and Ramsey's theorem  via Fra\"{\i}ss\'e's theorem on chainability (Theorem \ref{theo:chainability}) given below; the proof of item (f) is implication $(iii)\Rightarrow (i)$ of Theorem 2.25 p.17 of \cite {pouzet-thiery}. It uses properties of $Ker(R)$, the \emph{kernel} of $R$ (the set of $x$ of the base $E$ of $R$ such that $\age (R_{-x})$ is distinct from $\age (R)$).  
\begin{theorem} \label{theo:chainability} (Fra\"{\i}ss\'e)
Let $R= (E,(\rho_i)_{i\in I})$ be a relational structure on an infinite set $E$, $F$  a finite subset of $E$ and $\leq$  a linear order on $E\setminus F$. Then for each finite subset $I'$ of  $I$ there is an infinite subset  $X$ of  $E\setminus F$ such that $X$ is a chainable part of the reduct $R^{I'}_{\restriction F\cup X}$ (and the linear order $\leq$).  
\end{theorem}

  From Item (c) of Proposition \ref{prop: strong monomorphic component}, the existence of a finite monomorphic decomposition is equivalent to the existence of a finite strongly monomorphic decomposition; this is also equivalent to the existence of a linear order on $E$ and a partition of $E$ into finitely many intervals such that every partial map which preserves the order on each interval  is a local isomorphism of $R$.

\medskip

\medskip

We now come to a key tool we will use to estimate the number of siblings.
 
\begin{lemma}\label{lem:size component} Let $R$ be a relational structure with domain $E$  and $n\in \N$.  Then: 
\begin{enumerate}
\item The equivalence relations defining the monomorphic components of $R$ are preserved by every member of $Aut(R)$.
\item  If the number of orbits of singletons w.r.t. $Aut (R)$ is finite then the set $S$ of  integers $k$ such that some monomorphic component has cardinality $k$ is finite.  
\item If $R$ is $1$-homogeneous (that is, two elements $x, y$ such that $R_{\restriction \{x\}}$ and $R_{\restriction \{y\}}$ are isomorphic belong to the same orbit), then    the orbit of any $x\in E$ is a  union of monomorphic components of $R$, and all those components have the same cardinality. 
\item If $R$ is prehomogeneous, then  every  infinite monomorphic component is contained in the orbit of some singleton.
\item If $R$ is homogeneous, then the equivalences $\simeq_R$ and $\simeq_{Aut(R)}$ coincide. 
\end{enumerate}
\end{lemma}

\begin{proof} (1) Being definable (by infinitary formulae), the monomorphic decomposition is preserved under automorphisms. (2) If  two elements are in the same orbit of $Aut(R)$, then  the monomorphic component  containing $x$ and the monomorphic component containing $y$  have  the same size, hence (2) follows. (3) If two elements  are in the same monomorphic component, the restriction of $R$ to these elements are isomorphic;  if  $R$ is $1$-homogeneous, then there is some automorphism carrying one onto the other. (4) Let $x$, $y$ be in the same monomorphic component $X$. Since $R$ is prehomogeneous, there is some finite set $F_x$ containing $\{x\}$ such that every local isomorphism defined on $\{x\}$ which can be extended to $F_x$ can be extended to an automorphism. Now, since $X$ is infinite, then by Proposition \ref{prop: strong monomorphic component} $X$ is strongly monomorphic  and hence chainable over $E$ by some chain $(X, \leq)$. Set $F'_x= F_x\cap X$. Now, every local isomorphism of $(X, \leq)$ defined on $X$ and extendable by the identity on $E\setminus X$ will carry $x$ onto some $x'$ belonging to the same orbit, and the set $S_x$ of elements of $X$ which  cannot be attained from $x$ in this manner (if any) is by chainability the union of an initial interval and a final interval of $X$ whose size is at most $\vert F'_x\vert-1$. The same reasoning with $y$ in place of $x$ yields a set $F'_y$ of size at most $\vert F'_y\vert-1$. Since those sets are finite and $X$ is infinite, there are elements which can be reached from $x$ and $y$, and  hence $x$ can the transformed to $y$ by some automorphism as required. (5) The fact that $\simeq_{Aut(R)}$ is included in   $\simeq_R$ holds with no condition on $R$; the homogeneity of $R$ is used for the converse.   \end{proof}

\begin{remark}
Consider, as a comparative example, the direct sum of infinitely many copies of a $2$-element chain. It  is uniformly prehomogeneous, and the automorphism group has two orbits of singletons: the set of maximal elements and the set of minimal elements. The monomorphic components are the $2$-element chains, and none is contained in  an orbit. 
\end{remark}

We now revisit the action of a group on a set. Let $G$ be a permutation group acting on a set $E$. For $A$ a subset of $E$, we denote by $G_{\underline {A}}$, resp. $G_{A}$,  the   pointwise, resp. setwise stabilizer  of $A$. If $G$ leaves $A$ globally invariant (i.e., $G=G_{A}$), we  set $G{\restriction A}= \{\sigma {\restriction A}: \sigma\in G\}$. 

\medskip

We first deal with prehomogeneous structures.

\begin{proposition} \label{compprehomogeneous}
 Let  $R$ be  a  prehomogeneous structure on a countable set $E$, $G=Aut(R)$, and let $A$ be an infinite monomorphic component of $R$ with $B=E \setminus A$ its complement. Then 
\begin{enumerate}[{(a)}]
\item $G_{B} {\restriction A}$ is a  monomorphic group; 
\item If $G$ is oligomorphic, then $G_{\underline B}{\restriction A}$ is also oligomorphic,  and 
 there  is a dense linear order on $A$ such that for $C= (A, \leq)$,  $Aut(C)\subseteq G_{\underline B}{\restriction  A}$ and hence $G_{\underline B}{\restriction  A}$ is monomorphic.
\end{enumerate} 
\end{proposition}

\begin{proof} 
$(a)$. We must show that for every integer $n\geq 1$, if $F_1$ and $F_2$ are two $n$-element subsets of $A$, then there is some $\sigma \in G_{B}{\restriction A}$ which carries $F_1$ onto $F_2$. Let $\overline F_i$, finite, containing  $F_i$  be such that local embeddings  defined on $F_i$ which extend to $\overline F_i$ extend to automorphisms of $R$.  Since $A$ is an infinite monomorphic component, hence strongly monomorphic by Proposition \ref{prop: strong monomorphic component}, there is a linear order $\leq$ on $A$ such that finite local isomorphisms of $(A, \leq)$ extend by the identity on $B$ to local isomorphisms of $R$. Since $A$ is infinite, we may find an $n$-element subset $F$ in $A$ such that, via order preserving mappings on $\overline F_i \cap A$, $F_i$ is carried to $F$ by an isomorphism which extends to a local isomorphism fixing $B$ pointwise (and in particular $\overline F_i\cap B$). Since $R$ is prehomogeneous, this provides an automorphism $\sigma_i$ which carries $F_i$ to $F$ for $i=1,2$. Now, $\sigma_2^{-1}\circ \sigma_1$ carries $F_1$ to $F_2$ and is an automorphism. Since automorphisms must preserve the equivalence relation $\simeq_R$ and $F_i\subseteq A$, for $i=1,2$, this automorphism fixes $A$ set wise. Hence it fixes $B$ set wise, that is belongs to $G_{B}$.    Without invoking a stronger condition, e.g.,  oligomorphic action as in $(b)$,  we have not been able to show that there is an automorphism carrying $F_1$ on $F_2$ and fixing $B$ pointwise.

$(b)$ Once we know that $G_{\underline B}{\restriction A}$ is oligomorphic, the existence of a dense order on $A$ follows from Theorem \ref{freely rational}. In fact, we prove directly the existence  of a dense order as follows. On each infinite component $A_i$, we may put a linear order $\leq_i$ in such a way that the local isomorphisms of $(A_i, \leq_i)$ extended by the identity on the complement of $A_i$  are local isomorphism of $R$ (Proposition  \ref{prop: strong monomorphic component}). Extend each infinite component $A_i$ to a set $A'_{i}$ and $\leq_i$ to a dense order $\leq'_i$ in such a way that for distinct $i$'s the $A'_i$'s  are disjoint. We may extend $R$ to a relation $R'$ on $E'= E\cup \bigcup_i A'_i$ in such a way that any $1-1$ map of finite domain $F \subseteq E'$ which sends each $A'_i\cap F$ into $A_i$ and respect the order and fixes all other elements is a local isomorphism from $R$ into $R'$. The extension $R'$ has the same age as $R$. Since $R$ is prehomogeneous, $R'$ is a $1$-extension of $R$. Furthermore, if $R''$ is an extension of $R'$ with the same age this is a $1$-extension, that is  $R'$ is existentially closed. Since $G$ is oligomorphic, $R$ is the unique countable existentially closed structure  for its age, hence $R'$ is isomorphic to $R$. Clearly, $Aut(A'_i , \leq'_i)\subseteq  Aut (R'_{\underline B_i'}{\restriction  A'_i})$. Any isomorphism will transform the $A'_i$'s into the $A'_i$'s, and hence the image of the dense orders will give dense orders on the $A_i$'s with the required property. \end{proof}

Clearly, the fact that $G_{\underline B}{\restriction A}$ is monomorphic implies that $G_{ B}{\restriction A}$ is monomorphic. We do not know if the hypothesis of oligomorphy is really needed to prove the converse. In Proposition \ref{lem:notfull}, we show that the fact that $R$ is homogeneous suffices. %
%
%
%

\medskip

The following properties are folklore and straightforward.

 \begin{lemma}\label{lem:groupaction}
Let $G$ be a  group acting on a set $E$. If $G$ is closed in $\mathfrak S (E)$ then, for every subset $A$ of $E$, the groups $G_{A}$ and $G_{\underline A}$ are closed in $\mathfrak S(E)$ and   the  group $G_{\underline A}{\restriction (E\setminus A)}$ is  closed in $\mathfrak S(E\setminus A)$.  Provided that   $A$ or $E\setminus A$ is finite,  the group  $G_{A}{\restriction (E \setminus A)}$ is closed in $\mathfrak S(E\setminus A)$. 
\end{lemma}

\begin{remark}Without some condition on $A$, $G_{A}{\restriction (E \setminus A)}$ is not necessarily closed. For an example, let $R= (\Q, \leq, U)$ where $\leq$ is the natural order on the rationals and $U$ is a unary relation which divides the rationals into two dense sets. Let $G= Aut(R)$ and $A=\{ x\in \Q: U(x)= 1\}$. Then $G_{\restriction A}$ is different from $\overline{G_{\restriction A}}^{\mathfrak S}$, the closure of $G_{\restriction A}$  into $\mathfrak S$. Indeed, since $R$ is homogeneous, the latter group is equal to $Aut(R_{\restriction A})$. This group contains permutations which cannot extend to $\Q$; indeed if we choose $q$ with $U(q)=0$ and an irrational $r$ we may find  $\sigma \in Aut(R_{\restriction A})$ whose  extension  carries $q$ to $r$; this map $\sigma$ cannot be extended to $\Q$. In this example,  $G_{\restriction A}$ is monomorphic, hence oligomorphic, but not closed. The set $A$ is invariant under the action of $G$, but it is not a monomorphic component of $R$; in fact we may separate every pair of distinct elements $x$ and $y$ by some subset $F$ of $\Q$ with at most two elements.
\end{remark}

However, there is a powerful duality for a permutation group acting on two globally invariant sets.

\begin{lemma}\label{lem:group}  Let  $G$ be a permutation group acting on a set $E$ which is the union of two disjoint sets $A_0$ and $A_1$, leaving each of these sets globally invariant. Then the subgroup $H$ of $G$  generated by $\bigcup_{i<2}G_{\underline {A}_{i}}$ is a normal subgroup of $G$;  the group $G_{\underline {A}_{1-i}}{\restriction A_{i}}$ is a normal subgroup of $G{\restriction A_{i}}$ for every $i<2$;  and if   
 $H_i$ denotes the quotient of $G{\restriction A_{i}}$ by $G_{\underline {A}_{1-i}}{\restriction A_{i}}$, then $H_0$ and $H_1$ are isomorphic to the quotient of $G$ by $H$. \end{lemma}
\begin{proof}
  Let $\varphi_i: G\rightarrow G{\restriction A_i}$ defined by setting $\varphi_i(f)= f{\restriction A_i}$. Then $ker (\varphi)=G_{\underline {A}_{i}}$. Hence the quotient $G/ G_{\underline {A}_{i}}$ is isomorphic to $G_{\restriction A_i}$. Now  $G_{\underline {A}_{0}}$ and $G_{\underline {A}_{1}}$ commute, and in particular $H$ is isomorphic to the product $G_{\underline {A}_{0}}\times G_{\underline {A}_{1}}$.  It follows that $H$ is a normal subgroup of $G$ (for $f\in G$ and $h=h_0\circ h_1 \in H$ with $h_i\in G_{\underline {A}_{i}}$, ($i<2$), let  $f'=f^{-1}\circ h\circ f= f^{-1}\circ h_0\circ f\circ f^{-1} \circ h_1\circ f$; since $ G_{\underline {A}_{i}}$ is normal in $G$, it contains  $f^{-1}\circ h_i\circ f$,   thus $f'\in H$). Thus the quotient  $G/H$ is unambiguously defined. Next,  $G_{\underline {A}_{1-i}}{\restriction A_{i}}$ is a normal subgroup of $G{\restriction A_{i}}$. Indeed, we only need to check that $h^{-1}\circ G_{\underline {A}_{1-i}}{\restriction A_{i}}\circ  h\subseteq G_{\underline {A}_{1-i}}$ for every $h$ in $G{\restriction A_{i}}$. For that, let $g\in G_{\underline {A}_{1-i}}{\restriction A_{i}}$. Let $\underline g\in   G_{\underline {A}_{1-i}}$ such that $\underline g{\restriction A_{i}}=g$ and let $h'\in G$ such that $h'_{\restriction A_{i}}= h$. We have readily $h'^{-1}\circ \underline g\circ h'\in G_{\underline {A}_{1-i}}$. Hence $H_i$ is unambiguously defined.
  With the notation of Lemma \ref{lem:group}, we have: \begin{claim}\label{claim:group} For every $f\in G$, the following properties are equivalent: 
\begin{enumerate}[(i)]
\item  $f{\restriction A_0}$ extends to some $g_1\in G_{\underline A_1}$; 
\item  $f{\restriction A_1}$ extends to some $g_0\in G_{\underline A_0}$;
\item $f\in H$. 
\end{enumerate}
 \end{claim}
 \noindent{\bf Proof of Claim \ref{claim:group}.}
$(i)\Rightarrow (ii)$. Set $g_0= g_1^{-1}\circ f$. 
By the same token we have $(ii)\Rightarrow (i)$. 
$(i)\Rightarrow (iii)$. We have $f=g_1\circ g_0$.
$(iii) \Rightarrow (i)$. Immediate. 
 \hfill $\Box$
\begin{claim}\label{claim:group2}
$\varphi_i^{-1}(G_{\underline {A}_{1-i}}{\restriction A_{i}})=H$. 
\end{claim}
 \noindent{\bf Proof of Claim \ref{claim:group2}.}
By symmetry it suffices to prove the case $i=0$.  Let $h\in H$; then $h=h_0\circ h_1$ with $h_i\in G_{\underline {A}_{i}}$, ($i<2$). Hence, 
$\varphi_{0}(h)=\varphi_{0}(h_0\circ h_{1})= \varphi_{0}(h_0)\circ \varphi_{0}(h_{1})=\varphi_{0}(h_1)\in 
G_{\underline {A}_{1}} {\restriction A_{0}}$. Thus, $H \subseteq \varphi_0^{-1}(G_{\underline {A}_{1}}{\restriction A_{0}})$. Conversely, let $f\in \varphi_0^{-1}(G_{\underline {A}_{1}}{\restriction A_{0}})$. Then, $f{\restriction A_{0}}=\varphi_{0}(f) \in G_{\underline {A}_{1}}{\restriction A_{0}}$. Hence, $f{\restriction A_{0}}$ satisfies $(i)$ of Claim \ref{claim:group}, and hence satisfies $(iii)$ as well and $f\in H$. Thus $\varphi_0^{-1}(G_{\underline {A}_{1}}{\restriction A_{0}})\subseteq H$. Consequently, $\varphi_0^{-1}(G_{\underline {A}_{1}}{\restriction A_{0}})=H$, as claimed.
\hfill $\Box$

From Claim \ref{claim:group2} follows that  the quotient $G/H$ is isomorphic to the quotient $H_i$  of $G{\restriction A_{i}}$ by $G_{\underline {A}_{1-i}}{\restriction A_{i}}$. This  proves the lemma. 
\end{proof}

\medskip

We are now in a  position to better describe closed permutation groups with an infinite monomorphic component.

\begin{lemma} \label{lem:notfull} Let  $G$ be a closed permutation group acting on a countable set. Suppose there is an infinite $G$-monomorphic component $A$, and let $B$ be its complement. Then 
\begin{enumerate}[{(a)}]
\item $G_{\underline B}{\restriction A}$, $G_{B} {\restriction A}$ and its closure $\overline {(G_{B} {\restriction A})}^{\mathfrak S}$ are monomorphic groups;
\item  The quotient  of $G_{B}{\restriction A}$ by $G_{\underline {B}}{\restriction A}$ has at most two elements;
\item When that quotient has size $2$, and  $C=(A,\leq)$ is a dense linear order  such that  $Aut(C)\subseteq G_{\underline B}{\restriction  A}$,  then $G_{\underline {B}}{\restriction A}$ and $\overline {(G_B {\restriction A})}^{\mathfrak S}$ are either respectively  equal to  $Aut(C)$ and $Aut(B_C)$,  or else to $Aut(T_C)$ and $Aut(D_C)$.   
\end{enumerate} 
\end{lemma}
\begin{proof} $(a)$. Let $R$ be an homogeneous structure such that $Aut(R)= G$. Then $A$ is an infinite monomorphic component of $R$. According to Proposition \ref{prop: strong monomorphic component},  this is a strong monomorphic component of $R$. It follows  that, for every finite subset $B'$ of $B$, every integer $n$,  any two  $n$-elements subsets of $A$ are in the same orbit of  $G_{\underline B'} \restriction A$. Hence, each group  $G_{\underline B'} \restriction A$ is monomorphic.  Next, observe that  $G_{\underline B}\restriction A= \bigcap\{\overline {G_{\underline B'} \restriction A}^{\mathfrak S}: B'\in [B]^{<\omega}\}$.  The inclusion $G_{\underline B'}\restriction A\subseteq  \bigcap\{\overline {G_{\underline B'} \restriction A}^{\mathfrak S}: B'\in [B]^{<\omega}\}$ follows immediately from the obvious inclusions $G_{\underline B}\restriction A \subseteq G_{\underline B'} \restriction A\subseteq \overline {G_{\underline B'} \restriction A}^{\mathfrak S}$. The reverse inclusion is immediate: let $\sigma$ be in the above intersection, then  $\sigma$ extended by the identity on $A$ belongs to $G$,  that is $\sigma \in G_{\underline B}\restriction A$ (indeed, for every finite subset $F$ of $A$ and $B'$ finite in $B$, some $\tau \in G_{B'}\restriction A$ coincide with $\sigma$ on $F$). The groups $\overline {G_{\underline B'} \restriction A}^{\mathfrak S}$ are monomorphic and closed. Due to  Cameron's theorem, there is no infinite descending sequence of such groups, (cf Lemma  \ref{lem:nodescending}). Hence, $G_{\underline B}\restriction A= \overline {G_{\underline B'} \restriction A}^{\mathfrak S}$ for some  $B'\in [B]^{<\omega}\}$. This prove that $G_{\underline B}{\restriction A}$ is a closed monomorphic group.  Being overgroups  of that group, the groups  $G_{B} {\restriction A}$ and its closure are also monomorphic. 

$(b)$ and $(c)$. As said, the group $G_{\underline B}{\restriction A}$ is closed (a fact which follows directly from Lemma \ref{lem:groupaction}). By Theorem \ref {freely rational} there is a dense linear order $C= (A, \leq)$ such that $Aut(C)\subseteq G_{\underline B}{\restriction  A}$. 

For each integer $n\in \N$, the $n$-th member of the group sequence of $G_{\underline {B}}{\restriction A}$, resp. $\overline {(G_B{\restriction A})}^{\mathfrak S}$  is the group of permutations of an $n$ element subset $F$ of $A$ induced by members of $G_{\underline {B}}{\restriction A}$, resp. $\overline {(G_B{\restriction A})}^{\mathfrak S}$;  we will feel free to  denote these groups by $G_{\underline {B}}{\restriction F}$, resp. $\overline {(G_B{\restriction F})}^{\mathfrak S}$.  

Now let $F, F'$ be two  finite subsets of $A$ with $F\subseteq F'$. According to Lemma \ref{lem:group},  the  quotient  $H_F=G_{F}{\restriction B}/G_{\underline F}{ \restriction B}$ is isomorphic to the quotient $G_{B} {\restriction F}/G_{\underline B}{\restriction F}$ hence this quotient is finite. We claim that  it  can only decrease when the cardinality of $F$ increases; from Frasnay's result (Theorem \ref{thm:frasnay84}), the cardinalities of members of these two groups sequences is either $1$, $2$ or goes to infinity, and hence in our case this quotient can be only $1$  or $2$, and this will prove (b). When this quotient is constantly $1$,  the groups  $G_{\underline B}{\restriction A}$ and $\overline {(G_B{\restriction A})}^{\mathfrak S}$ are identical.  When this quotient is $2$, the only possibilities are those given in $(c)$.  

\medskip

It remains to verify our claim, thus let us consider two finite sets  $F\subseteq F' \subseteq A$.

\begin{claim}\label{claim:setstabilizer}  $G_{F'}{ \restriction B} \subseteq  G_{F}{ \restriction B}$.
\end{claim} 

\noindent{Proof of Claim \ref{claim:setstabilizer}}. Let $\sigma\in G_{F'}{ \restriction B}$. Then consider $F_1=\sigma (F)$ and let $\overline \sigma\in G_{F'}$ such that $\overline \sigma \restriction {B}= \sigma$. Since $Aut(C)\subseteq G_{\underline B}{\restriction  A}$, there is some $\theta \in G_{\underline B}$ such that $\theta (F_1)= F$. Let $\varphi'= \theta\circ \overline {\sigma}$. Then $\varphi'\in G_{F}$ and $\varphi'{\restriction B}= \overline{\sigma}{\restriction B}=\sigma$, hence
$\sigma\in G_{F}{ \restriction B}$. 
\hfill $\Box$

\begin{claim}\label{claim:pointstabilizer}  $G_{\underline F}{ \restriction B} = G_{\underline F'}{ \restriction B}$, provided that $\vert F\vert \geq 4$. 
In fact in this case $G_{\underline F}{ \restriction B} = G_{\underline{A}}{ \restriction B}$,
\end{claim} 

\noindent{Proof of Claim \ref{claim:pointstabilizer}}.  Clearly, $G_{\underline F'}{ \restriction B} \subseteq  G_{\underline F}{ \restriction B}$. Conversely, let $\sigma\in G_{\underline F}{ \restriction B}$ and let $\overline \sigma \in G_{\underline F}$ such that $ \overline \sigma {\restriction B}= \sigma$. Let $\theta$ be the extension of $\sigma$ (and $\overline{\sigma}$) by the identity on $A$. We prove that $\theta$ is an automorphism of $R$ from which follows that $\sigma \in  G_{\underline{A}}{ \restriction B}$, hence in $G_{\underline F'}{ \restriction B}$. 

Observe that  the group $\overline {(G_B{\restriction A})}^{\mathfrak S}$,  being monomorphic and closed, the classification given in Cameron's Theorem, asserts that any  homogeneous structure  $R'$ on $A$ with $Aut(R')= \overline {(G_B{\restriction A})}^{\mathfrak S}$ will  have the same local isomorphisms as  some relation which is at most $4$-ary. Local isomorphisms of $R'$ of finite domains are the  finite restrictions of members of $G_{B}{\restriction A}$. Hence we may suppose that $R'= R_{\restriction A}$, and  that if $\rho_i$ is a relation  occurring in $R$, then  each $n_i$-tuple $a\in \rho_i$ has  at most four components  in $A$.  Now, if these four components are in $F$, we will have $\theta(a) = \overline{\sigma}(a)\in \rho_i$ since $\overline{\sigma} \in G_{\underline F}$; if these four elements are not all in $F$, then  since $G_{\underline B}$ is monomorphic, it contains some  $\tau$ which sends these four  components into $F$; but now the previous case shows that $\theta(a)= \tau^{-1} \circ \theta \circ \tau (a)\in \rho_i$. 
\hfill  $\Box$

This completes the proof of Lemma \ref{lem:notfull}.

\end{proof}

By $(c)$ of this Lemma we get:
\begin{corollary}\label{cor:full symmetric}
Under the hypothesis of Lemma \ref{lem:notfull}, if $\overline {(G_{B}{\restriction A})}^{\mathfrak S}$ is  the full symmetric group on $A$, then $G_{\underline B} {\restriction A}$ is  also the full symmetric group on $A$. 

\end{corollary}

\begin{remark}
Corollary \ref{cor:full symmetric} also follows from the  Schreier-Ulam theorem \cite{schreierulam} on permutation groups (see also Scott \cite{scott} Page 305).

Indeed,  $G_{\underline {B}}{ \restriction A}$ is a  normal subgroup of $\overline {(G_{B}{\restriction A})}^{\mathfrak S}$. To see this let $g \in G_{\underline{B}}{\restriction A}$, and $h \in \overline {(G_{B}{\restriction A})}^{\mathfrak S}$. Now $h = lim_n \; h_n$ where $h_n \in G_{B}{\restriction A}$, and choose $\overline{h}_n \in G_B$ such that $ \overline{h}_n \restriction A=h_n$. Further choose $\overline{g} \in G_{\underline{B}}$ such that $\overline{g} \restriction A=g$. But now we have
$\overline{h}_n \circ \overline{g} \circ \overline{h}_n^{-1}=id_B \cup h_n\circ g \circ h_n^{-1} \in G_{\underline{B}}$, and since $lim_n \; h_n^{-1}\circ g \circ h_n = h \circ g \circ h^{-1}$we conclude that $h \circ g \circ h^{-1} \in G_{\underline {B}}{ \restriction A}$.

\medskip  

Now, the Schreier-Ulam theorem asserts  that the only proper normal subgroups of the symmetric group on a countable set 
are the group of permutations with finite support and the alternating subgroup. Neither of these groups is closed. 
Thus $G_{\underline {B}}{ \restriction A}$, being  closed, must be the full symmetric group.
\end{remark}

The following examples illustrate  quotients having two elements in Lemma \ref{lem:notfull}.

\begin{example}
Consider as a first example the countable set $E=\Q \times \{0,1\}$, naturally partitioned as $A=\Q \times \{0\}$,  $B=\Q \times \{1\}$. Define a quaternary relation $\rho(x,y,z,w)$ if $x,y \in A$, $z,w \in B$, and (abusing notation) satisfy $x<y$ iff $z<w$. Then $G=Aut(E,\rho)=\{ (f,g): f,g \in Aut(\Q), \mbox{ or } f,g \in Aut(B_\Q) \setminus Aut(\Q)\}$.
One easily verifies that in this case $G_{\underline B} {\restriction A} = Aut(\Q)$, and  $G_B{\restriction A} = Aut(B_\Q)$.

\medskip
Toward a second example, consider a chain $C=(E,\leq)$ and for an $n$-tuple $u=(u_1,u_2, \cdots, u_{n})$ of distinct elements of $E$, let $\sigma_u$ be the unique permutation of  ${\mathfrak S}_n$ such that $ u_{\sigma(1)} < u_{\sigma(2)} < \cdots < u_{\sigma(n)}$. Now for a subgroup $H\leq {\mathfrak S}_n$, form the $n$-ary relation $\rho_H=\{u \in E^n: \sigma_u \in H\}$. More generally consider two disjoint sets $A$ and $B$,  $H$ a subgroup of ${\mathfrak S}_n \times {\mathfrak S}_n$, and define a relation $\rho_H$ on $A \cup B$ by $\rho_H=\{(u,v) \in A^n \times B^n: (\sigma_u,\sigma_v) \in H\}$. Then $H=T_3^2 \cup (D_3\setminus T_3)^2$ is a group, and taking $A$ and $B$ as two disjoint copies of the rationals and $G=Aut(A \cup B, \rho_H)$ we obtain that  
$G_{\underline B} {\restriction A} = Aut(T_\Q)$, and  $G_B{\restriction A} = Aut(D_\Q)$.

\medskip
In that setting the first example can be restated using $H=\mathfrak{I}(2)^2 \cup (\mathfrak{S}(2) \setminus \mathfrak{I}(2))^2$.

\end{example}

\section{Finite monomorphic decomposition}\label{Finite monomorphic decomposition}

We  prove the following result.

\begin{theorem}\label{thm:monomorphic-siblings} If a countable relational structure $R$ is prehomogeneous and has a finite monomorphic decomposition,  then it has one or $2^{\aleph_0}$ siblings.  
\end{theorem}
 
 \begin{case} $R$ has just one component. 
\end{case}
If so the conclusion follows from Theorem \ref{thm:siblingmonomorphic}: $\sib(R)$ is one if $Aut(R)$ is the symmetric group, and $2^{\aleph_0}$ otherwise.

\begin{case} $R$ has several monomorphic components. 
\end{case}

In this case the result follows from Theorem \ref{prop:keyfinite} below. Indeed suppose no infinite component is as in Theorem \ref{prop:keyfinite}. Then taking the partition whose classes are the infinite components and then singletons shows the structure is finitely partitioned. It is clear such structures have only one sibling.


%
%

\begin{theorem}\label{prop:keyfinite}
Let  $R$ be a countable structure which is prehomogeneous and such that $G=Aut(R)$ is oligomorphic. If $R$ has  an infinite monomorphic component $A$ which is not an strongly indiscernible subset of $R$  then $R$ has $2^{\aleph_0}$ siblings. 
\end{theorem}

\begin{proof} The fact that  $A$  is not an indiscernible subset of $R$ means that $G_{\underline B}{\restriction A}$ (where $B$ is the complement of $A$) is not the full symmetric group on $A$. According to Proposition  \ref{compprehomogeneous},   $G_{\underline B}{\restriction A}$ and hence $ \overline {G_{B}{\restriction A}}^{\mathfrak S}$ are monomorphic groups; their structure is given by Lemma \ref{lem:notfull}.  According to Corollary  \ref{cor:full symmetric}, since $G_{\underline B}{\restriction A}$ is assumed not to be the full symmetric group,  $\overline {G_{B}{\restriction A}}^{\mathfrak S}$ is not the full symmetric group either. 
According to Theorem \ref{thm:siblingmonomorphic}, any structure $S$ on $A$ with $Aut(S)= \overline {G_{B}{\restriction A}}^{\mathfrak S}$ has $2^{\aleph_0}$ siblings. That is there are  $2^{\aleph_0}$ subsets $(A_\alpha)_{\alpha < 2^{\aleph_0}}$ of $A$ such that for each $\alpha, \alpha'$ there is a $\overline {G_{B}{\restriction A}}^{\mathfrak S}$-embedding of $A$ into $A_{\alpha}$, and no 
$\overline {G_{B}{\restriction A}}^{\mathfrak S}$-embedding of $A_\alpha$ onto $A_{\alpha'}$. 

\begin{claim} \label{claim:alphasibling}
Each restriction $R_\alpha = R \restriction E_\alpha$, where $E_\alpha= B \cup A_\alpha$, is a sibling of $R$.
\end{claim}

\noindent {\bf Proof of claim \ref{claim:alphasibling}}
Since $G=Aut(R)$ is oligomorphic,  there is a dense linear ordering $C=(A,\leq)$ on $A$ such that 
$Aut(C) \subseteq G_{\underline B}{\restriction A}$ (Proposition \ref{compprehomogeneous} b.). Since there is some $ \overline {G_{B}{\restriction A}}^{\mathfrak S}$-embedding $\sigma$ of $A$ into $A_{\alpha}$, the members of the indicative sequence of the bichain $(C, C_{\sigma^{-1}})$ (where $C_{\sigma^{-1}}=(A,\leq_{\sigma^{-1}})$ and $x \leq_{\sigma^{-1}} y$ iff 
$\sigma(x) \leq \sigma(y)$) are termwise included into the group sequence of $ \overline {G_{B}{\restriction A}}^{\mathfrak S}$ . If $C_{\sigma^{-1}}$ is scattered, then by Lemma \ref{lem:scatteredbichain}  this indicative sequence is $\mathfrak S$, and hence the group sequence of $ \overline {G_{B}{\restriction A}}^{\mathfrak S}$  is $\mathfrak S$ which is excluded. Consequently  $C_{\sigma^{-1}}$ is non-scattered, that is $C_{\sigma(A)}=(\sigma(A),\leq)$ is  non-scattered, and there is an embedding of $C$ into $C_{\sigma(A)}$, this embedding extended by the identity on $B$ is an embedding of $R$ into $R_\alpha$. This proves the claim. \hfill $\Box$

\medskip

Now let $\Gamma = \{ \{\alpha, \beta\}: R_\alpha \cong R_\beta \}$, and $\kappa$ the number of monomorphic components of $R$.

\begin{subclaim} \label{claim:boundsibling} $\vert [C]^2\vert \leq \kappa^2$ for each isomorphism equivalence class $C$ of siblings of $R$. 
\end{subclaim}

\noindent {\bf Proof of subclaim \ref{claim:boundsibling}} The structures $R$, $R_{\alpha}$ and $R_{\beta}$ have the same induced monomorphic decomposition. Hence, if $\alpha$ and $\beta$ are equivalent, an isomorphism of $R_{\alpha}$ onto $R_{\beta}$ induces a permutation of the classes of the decomposition. Such a  permutation sends $A_{\alpha}$ to some class and  $A_{\beta}$ to another one. If $\vert [C]^2\vert > \kappa^2$ then two pairs $\{\alpha, \beta\}$ and $\{\alpha', \beta'\}$ would be sent to the same pair of classes of the decomposition, but then $A_{\alpha}$ would be sent onto $A_{\alpha'}$. This map, being a $G_{B}{\restriction A}$-embedding from $A_\alpha$ to $A_{\alpha'}$, would be a $ \overline {G_{B}{\restriction A}}^{\mathfrak S}$-embedding from $A_\alpha$ to $A_{\alpha'}$, but there is none. \hfill $\Box$

\medskip

Since $\kappa$ is countable, there are $ 2^{\aleph_0}$ inequivalent elements, that is $2^{\aleph_0}$ siblings. This proves Theorem \ref{prop:keyfinite}.
\end{proof}

\medskip

Theorem \ref{thm:monomorphic-siblings} appears rather weak to us, we make the following conjecture.

\begin{conjecture}  
Under the  assumption that  a relational structure $R$ has a finite monomorphic decomposition, $R$ has one or infinitely many siblings.
\end{conjecture}

 We proved that it holds if the structure is a chain \cite{LPW}. One would need to extend  this conclusion to the case of an infinite monomorphic relational structure $R$; such a structure is chainable.  Next, one must  go from a monomorphic structure to one admitting a finite monomorphic decomposition. 

\section{Structures with no finite monomorphic decomposition}\label{section:nofinitemonomorphicdecomposition}

In this section, we prove the following result:

\begin{theorem}\label{thm infinitely many}Let $R$ be  a countable prehomogeneous relational structure such that $Aut(R)$ is  oligomorphic. 
\begin{enumerate}[(a)]
\item If $R$ has infinitely many monomorphic components then $sib(R)$ is infinite; 
\item If $R$ has infinitely many infinite monomorphic components then $\sib(R)=2^{\aleph_0}$.
\end{enumerate}
\end{theorem}

We prove  $(a)$ in Subsection \ref{proof of a}.  For that, we  show in the next subsection  that $R$ has a $1$-extension $R'$  such that $E'\setminus E$, the difference of the two domains,  is an infinite monomorphic part and for which the component of $R'$ containing it meets $E$ on a finite set (Lemma \ref{good extension}). Then,  we show that the extensions of $R$ to  subsets of $E'\setminus E$ having finite distinct cardinalities provide distinct siblings (Proposition \ref{cor:manyextensions}).

We prove $(b)$ in Subsection \ref{proof of b}.  We prove that if $R$ has  infinitely many infinite classes,  one may  select  a countable union $D$ of infinite equivalence classes such that $R$ is embeddable into $R_{\restriction E\setminus D}$ (Lemma \ref{addinginfinitelymany}). Then,we apply  Lemma \ref{lem:infinitelymany}.

\subsection{Adding an infinite monomorphic part} 






\begin{lemma}\label{lem:trace}Let $R$ and $R'$ be two relational structures with  domains $E$ and $E'$ respectively. If $R'$ is an extension, resp. a $1$-extension, of $R$ then the partition of $E'$  into the monomorphic components of $R'$ induces a partition of $E$ into monomorphic parts, resp. monomorphic components of $R$. \end{lemma}

\begin{proof}
Let $(E'_j)_{j\in J}$ be the  monomorphic decomposition of $R'$ and let $(E'_j\cap E)_{j\in J}$ be the family of the induced blocks. Trivially, these sets  are monomorphic parts  of $R$, hence our first assertion holds. Furthermore,  the partition into those parts  is finer than the partition given by the monomorphic decomposition of $R$. To prove that these two partitions coincide whenever $R'$ is a $1$-extension of $R$, let $x$ and $y$ be two different blocks of the induced partition. They are inequivalent for $R'$, hence, there is a finite subset  $F'$ of $E'\setminus\{x,y\}$ witnessing that fact. Since  $R'$ is a  $1$-extension of $R$,  there is a local isomorphism from $R'$ to $R$ which fixes  $x, y$ and $F'\cap E$. The image $F$ of  $F'$ witnesses that  $x$ and $y$ are  inequivalent modulo $R$, that they  are into two parts  of the monomorphic decomposition. This proves our assertion.
\end{proof}

\begin{lemma}Let $R$ be a relational structure with domain $E$ and let  $R'$ be a $1$-extension of $R$ with domain $E'$ such that $E'\setminus E$ is an infinite  monomorphic part of $R'$. If the trace over $E$ of the component $C'$ of $R'$ containing $E'\setminus E$ is infinite then $R'$ is a $1$-extension of $R_{-C'}:= R_{\restriction E\setminus C'}$ and the monomorphic decomposition of $R_{-C'}$ is made of  components of $R$.  
\end{lemma}
\begin{proof}
Let $F$ be a finite subset of $E'$. We have to show that there is a local isomorphism $h$ of $R'$ that fixes $F\setminus C'$ and maps  $F\cap C'$ into $E'\setminus C'$. Since $E'\setminus E$ is infinite, it contains a subset $X$ with the same cardinality as $F\setminus C'$. Since $C'$ is an infinite monomorphic component, it is strongly monomorphic (cf. $(f)$ of Proposition \ref{prop: strong monomorphic component}),  hence there is a local isomorphism $f$ that fixes $F\setminus C'$ and maps $F\cap C'$ onto $X$. Since $R'$ is $1$-extension of $R$ and $C'\cap E$ is finite, there is a local isomorphism of $R'$ that fixes $(F\setminus C') \cup (C'\cap E)$ and maps $X$ on a subset of $E'\setminus E$. Since this subset is disjoint from $C'$, we may set $h:=  g\circ f$.

Since $R'$ is a $1$-extension of $R$ and of $R_{-C'}$, then according to Lemma \ref{lem:trace} the monomorphic decomposition of $R'$ induces the monomorphic decompositions of $R$ and of $R_{-C'}$. It follows that the monomorphic decomposition of $R_{-C}$ is made of  components of $R$. 
\end{proof}

\begin{lemma} \label{lem:many blocks}Let   $(E_i)_{i\in I}$ be the monomorphic decomposition into components of a relational structure $R$ with base  $E$. Let  $R'$ be a $1$-extension of  $R$ with base  $E'$ such that $E'\setminus E$ is a monomorphic part  of  $R'$. \\
Then either: 
\begin{enumerate}
\item There is some index $i$ such that $E_i\cup (E'\setminus E)$ is a monomorphic component of $R'$ and for every index  $j\not =i$,  $E_j$ is a monomorphic component of   $R'$; \\
or
\item  There is some non-negative  integer $k$,   $k\leq \vert E'\setminus E\vert$ such that if  $H'\subseteq E'\setminus E$, has  at least $k$ elements,  the monomorphic decomposition of  $R'_{\restriction E\cup H'}$ is made of  the $E_i$'s and of  $H'$. 
\end{enumerate}
\end{lemma}

\begin{proof}
We start with the following claim. 

\begin{claim}\label{claim1} Let  $z\in E'\setminus E$ and $R'_{z}= R'_{\restriction E\cup \{z\}}$. Then, either there is a unique index  $i$ such that   $E_i\cup \{z\}$ is a monomorphic component of $R'_{z}$, or   $\{z\}$ is a monomorphic component of $R'_{z}$.  
\end{claim}

\noindent{\bf Proof of Claim \ref{claim1}.}
 If there is some index $i$ and some $x_i\in E_i$ such that $z$ and  $x_i$ are equivalent modulo $R'_z$,  then we show that $i$ is unique. Indeed suppose on the contrary that we have $j\not =i$ and  $x_j\in E_j$ such that $z$ and  $x_j$ are equivalent modulo $R'_z$. It follows that $x_i$ and $x_j$ are equivalent modulo $R'_z$. This implies that there are equivalent modulo $R$, a contradiction. This proves the uniqueness of  $i$ if it exists.  Since $R'$ is a $1$-extension of $R$, $R'_{z}$ is a $1$-extension too, hence Lemma  \ref{lem:trace} applies. Thus, if there is no such  $i$, then $\{z\}$ must be  a  monomorphic component of $R'_{z}$. 
 \hfill $\Box$

%

With this claim, the proof of the lemma goes as follows.  According to Lemma \ref{lem:trace}, the decomposition of $R$ is induced by  the decomposition of $R'$. Hence, either $E'\setminus E$ union  some  $E_i$ (in fact a unique one) forms a component, or not. In this later case, we claim that there exists some integer  $k$ such that every finite subset  $H'$ with at least  $k$ elements of $E'\setminus E$ is a monomorphic component  of  $R'_{\restriction E\cup H'}$.

Indeed, let $\mathcal H$ be the set of  finite $H'\subseteq E'\setminus E$ such that $H'$ is not a monomorphic component of $R'_{\restriction E\cup H'}$. If there is some finite subset $H\not \in \mathcal H$ then for   $k= \vert H\vert$ we will have the conclusion of our claim. Suppose that every finite subset $H'$ of $E'\setminus E$ belongs to $\mathcal H$. For each finite $H'$ there is some index $i'$ such that $E_{i'} \cup H'$ is a monomorphic component of $R'_{\restriction E\cup H'}$. If $H'$ is non-empty it follows  from  Claim \ref{claim1} that this $i'$ is unique (pick any $z\in E'\setminus E$ and observe that $\{z\}$ cannot form a component of $R'_{z}$). If $i'$ depends on $H'$, that is there is some $H''$ and some $i''\not =i'$ with the same property, then for $H= H'\cup H''$ there will be no $i$ such that $E_i\cup  H$ is a monomorphic component  hence $H$ will be a monomorphic component of $R'_{\restriction E\cup H}$, contradicting our hypothesis on $\mathcal H$. Hence  $i'$ is independent of $H'$, meaning that there is  a unique index  $i$ such that  $E_i \cup H$ is a monomorphic component of  $R'_{\restriction E\cup H}$ for every finite subset $H$ of $E'\setminus E$. It follows then that   $E_i\cup (E'\setminus E)$ is a monomorphic part of $R'$ and, in fact, a component of $R'$. Indeed, if not, we may pick $x_i \in  E_i$, $z \in E' \setminus E$ and $F  \subseteq E'\setminus  \{x_i, z\}$ witnessing that there are not equivalent modulo $R'$;  setting   $H= F\cap (E'\setminus E)\cup \{z\}$ we will have that $x_i$ and $z$ are not equivalent modulo $R'_{\restriction E\cup H}$, which is impossible since $E_i\cup H$ is a component of $R'_{\restriction E\cup H}$. 
\end{proof}.

\medskip

As a consequence, we get:

\begin{proposition} \label{cor:manyextensions}Let $R$ be a relational structure with  domain  $E$ and  $S$ be the set of non-negative integers  $n$ such that  $R$ has \textbf{no} monomorphic component of size  $n$, and suppose that   $S$ is infinite. If $R$ has a   $1$-extension $R'$  such that $E'\setminus E$ is  an infinite monomorphic part of $R'$ and the trace over $E$ of the component $C'$ of $R'$ containing $E'\setminus E$ is finite then $R$ has an infinitely many   $1$-extensions   which are pairwise non isomorphic. 
\end{proposition}

\begin{proof}  Let $R'_{C'}= R'_{\restriction E'\setminus C'}$. We apply Lemma \ref{lem:many blocks} to $R'_{-C}$ and $R'$. Since $C'$ is a monomorphic component of $R'$ then for no monomorphic component $E_i$ of $R_{-C'}$ can $E_i \cup C'$
be a monomorphic component of $R'$, that is, Case 1 of Lemma \ref{lem:many blocks} cannot happen.

Thus the second case must hold. That is, there is some non-negative  integer $k$,   $k\leq \vert C'\vert$ such that for every  $H\subseteq C'$, with at least $k$ elements the monomorphic decomposition of  $R'_H= R'_{\restriction (E'\setminus C')\cup H}$ is made of  the $E_i$'s and of  $H$. For distinct values of $k=\vert H\vert$,  with $k\in S$ the $R_H$'s cannot be isomorphic,  otherwise the decomposition of a  $R_H$ will be carried over the decomposition of an $R_H'$. Taking for $H$ subsets of $C'$ containing $C'\cap E$ yields the desired conclusion.   
\end{proof}

We need  the following result 

\begin{lemma}\label{lem:adding a block} Let  $R= (E, (\rho_i)_{i\in I})$ be a relational structure of signature $\mu= (n_i)_{i\in I}$ on an infinite set $E$ and $A$ be an infinite subset of $E$.  If the profile is finite, 
then,  on any  superset  $E'$ of $E$ such that $E'\setminus E$ is infinite, there is some   extension $R'$ of  $R$ 
 such that:
 
   \begin{enumerate} [(a)]
   
\item $E'\setminus E$ is a strong monomorphic part  of $R'$ and 
 
\item  for every finite subset $F$ of $E'$ there is some local isomorphism of $R'$ which fixes $E\cap F$ and maps $F\setminus E$ into $A$.

\end{enumerate}
\end{lemma}

In the case $A=E$, this is Lemme III-2.2.3 of \cite{pouzet79}. The  proof uses Theorem \ref {theo:chainability} of Fra\"iss\'e  and The Compactness Theorem of First Order Logic. In our case, the same proof applies.

We say that an extension $R'$ of $R$ as above is a \emph{good extension above $A$}.
\begin{lemma}\label{good extension}
Let  $R$ be prehomogeneous on a countable set $E$, $O$ be  an infinite orbit of a singleton w.r.t. $Aut(R)$ that meets infinitely many components of $R$, $A$ a subset of  $O$ such that $\vert A\cap C\vert = 1$ for each component $C$ of $R$ meeting $A$,  $R'$  an extension of $R$ to a superset $E'$ that is good above $A$ and $C'$  the component of $R'$ containing $E'\setminus E$.  Then  $C'\cap E$ contains at most one element and this element belongs to  $A$. \end{lemma} 

\begin{proof}Since $E'\setminus E$ is a monomorphic part of $R'$, it is included in a component of $R'$, say $C'$. Either  $C'= E'\setminus E$, that is, $E'\setminus E$ is a component, or not. In the first case $C'\cap E= \emptyset$ is a subset of $A$. In the later case, since $R'$ is a $1$-extension of $R$, the decomposition of $R$ into components is induced by the decomposition of $R'$, hence $C:= C'\cap E$ is a component of $R$. We prove first that $C \subseteq O$.  Suppose not. Let $b\in C\setminus O$.  Since $R$ is prehomogeneous, there is some finite set $F\subseteq E$ containing $b$ such that every local isomorphism $f$ of $R$ defined on $b$ that extends to $F$ to a local isomorphism of $R$ can be extended to an automorphism of $R$. Since $b\not \in O$ and  $O$ is an orbit, no local isomorphism $f$ can map $b$ into $O$ and extend to $F$. Since $C'$ is an infinite component of $R'$, it is a strongly monomorphic part, hence there is map $h$ from  $F\cap C'$ into $E'\setminus E$ that extends by the identity on $F\setminus C'$  to a local isomorphism of $R'$. Since $R'$ is a good extension above $A$, there is a local isomorphism $g$ of $R'$ that fixes $F\setminus C'$ and sends $h(F\cap C')$  into $A$. But then $g\circ h$ maps $b$ into $A$ thus into $O$. A contradiction. Suppose that  $C$ contains at least two elements. Since $C$ is a component, every  automorphism of $R$ sending some   element of $C$ into $C$ sends all the others into $C$. Since $R$ is  prehomogeneous, if $X$ is a $2$-element subset of $C$ there is some finite set $F\subseteq E$ containing $X$ such that every local isomorphism $f$ of $R$ defined on $X$ that extends to $F$ to a local isomorphism of $R$ can be extended to an automorphism of $R$. Furthermore, if $X'$ is an other $2$-element subset of $C$, we may find $F'$,  image of $F$ by some automorphism of $R$,  satisfying the same property.  Fix a two element subset $X$ of $C$. Since $C'$ is an infinite component of $R'$, it is a strongly monomorphic part, hence there is map $h$ from  $F\cap C$ onto $((F\cap C)\setminus \{a\})\cup \{b\}$ 
for  some   $a\in X\cap C$, $b\in E'\setminus E$,  that  extends by the identity on $F\setminus C$  to a local isomorphism of $R'$. Since $R'$ is a good extension above $A$, there is a local isomorphism $g$ of $R'$ that fixes $F\cup (A\cap C)$ and sends $h(F\cap C)$  into $A$. But, since $\vert A\cap C\vert =1$  then $g\circ h$ maps $a$ into $A\setminus C$ and the other elements of $F\cap C$ into $C$. A contradiction. \end{proof}
\begin{remark}
From this lemma it follows that $E'\setminus E$ is a component of $R'$ whenever the components of $R$ meeting $O$ are non-trivial. 
But  it is not true in general that $E'\setminus E$ is a component of $R'$. For an example,  take for $R$ the Rado graph, fix a vertex, say $a$, add an infinite independent set, say $H$, and for every $x$ in $R$, if $x$ is  joined to $a$ by an edge, join $x$ to every vertex of $H$, otherwise $x$ is  joined to no vertex of $H$. Then the resulting graph $G'$ is a $1$-extension over $E$ (as well as  the non neighbour of $a$) and $H\cup \{a\}$ is a component of $R'$. However, there are extensions of the Rado graph for which $E'\setminus E$ is a component. 
\end{remark}
\begin{problem}\label{lemma:key}
Is it true that a   countable  prehomogeneous structure $R$ with  infinitely many components  and $Aut(R)$ oligomorphic has a $1$-extension $R'$ with $E' \setminus E$ an infinite component?  \end{problem}

\subsection{Proof of $(a)$ of Theorem \ref{thm infinitely many}}\label{proof of a} 
Since $Aut(R)$ is oligomorphic, it has only finitely many orbits of singletons. One, say $O$, meets infinitely many  classes. Let $A$ be a subset $A$ of  $O$ such that $\vert A\cap C\vert = 1$ for each component $C$ of $R$ meeting $A$.  Since $Aut(R)$ is oligomorphic, the profile of $R$ is finite, hence Lemma \ref{lem:adding a block} applies and there is some extension $R'$ of $R$ above  $A$. Then, according to Lemma \ref{good extension}, then either   $E'\setminus E$ is a component of $R'$ or the component $C'$ of $R'$ containing $E'\setminus E$ is made of $E'\setminus E$ and a singleton belonging to $O$. Since  $C'\cap E$ is a finite component of $R$ we may then apply Proposition  \ref{cor:manyextensions}. 
\hfill $\Box$

\subsection{Adding infinitely many monomorphic components. A proof of $(b)$ of Theorem \ref{thm infinitely many}}\label{proof of b}
\begin{lemma} \label{addinginfinitelymany} If $R$ is countable, uniformly prehomogeneous, with infinitely many infinite equivalence classes of $\simeq_R$,  one may  select  a countable union $D$ of infinite equivalence classes   such that $R$ is embeddable into 
$R_{\restriction E\setminus D}$
\end{lemma}

\begin{proof} We prove  a slightly different statement. Namely, under the conditions of the lemma, $R$  has an extension $R'$ which is isomorphic to $R$ and such that $E' \setminus  E$, the difference of their bases,  contains infinitely many infinite components of $R'$.  
For that,  we will use the diagram method due to Robinson and 
apply the compactness theorem of first order logic. 

We enumerate the  elements of $E$ to form  a sequence  $a_n$, $n<\omega$. To the language of $R$ we add these elements as constants and we also add a new infinite set of constants  $c_{i, j}$, $i,j<\omega$. We add some sentences and  we prove that they form a consistent set, and thus compactness of first order logic ensures  that there is some countable model. Due to our choice of sentences, this model will be a model of the universal theory of $R$, hence it will extend to a copy $R'$of $R$. The constants $c_{i,j}$ will satisfy the following three properties: 
\begin{enumerate}
\item   $c_{i,j}\not =c_{i',j'}$ when $(i,j)\neq (i',j')$; 
\item   $c_{i,j}\simeq_{R'} c_{i',j'}$ if and only if $i=i'$; 
\item   $c_{i,j}\not \simeq_{R'} e$ for $e\in E$.  
\end{enumerate}
Hence,  $E' \setminus  E$, the difference of their bases,  contains infinitely many infinite components of $R'$ which avoid $E$.   %
%
The sentences we add fall into three categories.  

$a)$ Those of the  diagram of $R$. That is, we add the sentences of the form  $\rho_i(a_{i, 1}, \dots, a_{i,  m_i})$ for every  $(a_{i,1}, \dots, a_{i, m_i})\in \rho_i$, the sentences of the form  $\neg \rho_i(a_{i,1}, \dots, a_{i,m_i})$ for every $(a_{i,1}, \dots,a_{i,m_i})\not \in \rho_i$  and the sentences of the form $a_i\not =a_j$ for every  $i \not =j$. Clearly, any model of the diagram will be an extension of $R$. 

$b)$ The sentences of the form $\forall x_1, \dots, \forall x_p \neg F(x_1, \dots, x_p)$ where $F$ is a quantifier free formula in the language of $R$ describing a finite reduct which cannot be embedded in  $R$.  These sentences, added to the previous one, form a consistent set;  indeed $R$ is a model.  Furthermore, any model $R'$ is an extension, and in fact a  $1$-extension, hence a model of the universal theory of $R$. 

$c)$ Sentences  expressing that $(1)$, $(2)$ and $(3)$ hold.
%

For that,  we note that there is an integer $k$ such that $\simeq_{\leq k, R}$ and $\simeq_R$ coincide (Lemma \ref{lem:definability}). Since the profile of $R$ is finite,  it follows that there is an existential formula $F(x,y)$ (using at most $k$ quantifiers) such that $a\not \simeq_R b$ if and only if $F(a,b)$ holds in $R$. We  add to the diagram of $R$ the sentences  $F(c_{i,j},c_{i'j'})$ for $i<i'$,  $F(a_i, c_{i',j})$ for all $i, i'$ and $\neg F(c_{i,j}, c_{i,j'})$ for all $i,j,j'$.

This set of sentences added to the previous ones is consistent. Indeed, taking finitely many, they will determine a finite subset $A$ of $E$ and a finite subset $C$ of the $(i,j)$'s and will define an equivalence relation on $C$. Since $R$ contains infinitely many infinite components,  there are infinitely many that are disjoint from $A$, hence we may select in these components  elements reproducing the structure of the equivalence relation over $C$ to obtain the consistency of this finite set of sentences.   As  noted above, the compactness theorem of first order logic will give a copy $R'$ of  $R$ extending $R$. In  that copy, two elements $a,b$   satisfy  $a\simeq_{R'}b$ if and only if  $\neg F(a,b)$. Since  $F(x,y)$ is an existential formula,  the  $c_{i,j}$'s will not be $\simeq_{R'}$-equivalent and not  $\simeq_{R'}$-equivalent to any element of $E$.  In that copy,  the union $D$ of the equivalence classes of the $c_{i,j}$'s is disjoint from $R$. 
\end{proof}

With Lemma \ref{addinginfinitelymany} and Lemma \ref{lem:infinitelymany} we get that  $R$ has $2^{\aleph_0}$  siblings. Hence,  $(b)$ of Theorem \ref{thm infinitely many} holds.  

\section{Proof of Theorem \ref{thm:main}}\label{section:maintheorem}

Reassembling  Theorem \ref{prop:keyfinite}, Theorem  \ref{thm:monomorphic-siblings} and Theorem \ref{thm infinitely many}, we get: 

\begin{theorem}\label{thm:main2} Let $R$ be  a countable prehomogeneous relational structure such that $Aut(R)$ is  oligomorphic. 
Then $R$ has $2^{\aleph_0}$  siblings if
 $R$ has some infinite monomorphic component which is not an indiscernible set of $R$, or infinitely many infinite components.  If not then $R$ has one sibling provided that  $R$ has finitely many components, and infinitely many siblings otherwise. 
%
%
\end{theorem}

Theorem  \ref{thm:main} follows.

%

\section{Conclusion} \label{section:thelast}

%
%
%
\subsection{A possible improvement of Theorem \ref{prop:keyfinite}.}   
 Let us recall that a relational structure $R$ with base $E$ is \emph{cellular} \cite{schmerl} if there is a finite subset $F\subset E$  and an enumeration $(a_{(x,y)})(x,y)\in V \times L $ of the elements of $E \setminus  F$ by
a set $V \times  L$, where $V$ is finite  such that for every bijective map 
 $f$ of $L$ the map $(1_V, f)$ extended by the identity on $F$ is a local isomorphism
of $R$ (the map $(1_V, f)$ is defined by $(1_V , f)(a_{(x, y)}) := a_{(x, f(y))}$).
Note that a finitely partitionnable structure is cellular, but the converse does not hold. 

The age $\mathcal A$ of a cellular structure $R$ is well-quasi-ordered (w.q.o, for short), that is every infinite sequence $(S_n)_{n\in \N}$ of members of $A$ contains an increasing sequence w.r.t embeddability. In fact, the set $\mathcal A_{[m]}$ of structures $S\in \mathcal A$ with $m$ unary relations added, is also w.q.o. for every $m\in \N$. It follows from Th\'eor\`eme 3.4.  p.697 of \cite {pouzet} that if for an age $\mathcal A$, the set $\mathcal A_{[m^{-}]}$ of structures $S\in \mathcal A$ with $m$ constants  added, is  w.q.o. for every $m\in \N$ then there is a  uniformly prehomogeneous structure with age $\mathcal A$ (we do not know if these two w.q.o. conditions are equivalent). In  particular if $R$ is cellular, some $R'$ equimorphic to $R$ is uniformly prehomogeneous (and cellular). 
If $R$ is cellular then $\sib(R)$ is at most countable; this is a straightforward consequence of a result of \cite {macpherson-pouzet-woodrow} (see below).

Under this setting we propose the following problem.
\begin{problem}\label{prob:improvement}
Let $R$ be a countable and $\aleph_{0}$-categorical relational structure. Prove that  either $\sib(R)=1$, $\aleph_0$ or $2^{\aleph_0}$. Furthermore, show that $\sib(R)\leq\aleph_0$ if and only if  $R$ is cellular. \end{problem}
%
Note that  Theorem \ref{prop:keyfinite} does not give  the value of $\sib(R)$ when $R$ has infinitely many finite components and finitely many infinite components which are strongly indiscernible. We know that $\sib(R)$ is infinite, but there are examples such  that  the number of siblings is countable and some for which it is the continuum. 
Note that if our problem has a positive answer, then $\sib(R)= 2^{\aleph_0}$ whenever all the components are singletons. 

\subsection{A possible extension to universal structures} 
A countable structure  $R$ is \emph{universal for its age} if every countable structure with the same age embeds into $R$.
\begin{problem}\label{prob:extension}
Get the same conclusion as in Problem \ref{prob:improvement} under the weaker requirement that $R$  is universal for its age  and the profile takes only  integer values.
\end{problem}

Some condition, e.g. that the profile of $R$ takes only integer values,  is necessary. Indeed, let $R_{\omega}$ be the relational structure made of a countable set and infinitely many distinct constants, such that the set not covered by the constants is infinite. With our definition of age, $R$ is  unique for its age, but not finitely partitionable. On the other hand, the universal theory $T_{\forall} (R)$ of $R$ has countably many countable models, namely $R_{0}, \dots R_{n}, \dots,  R_{\omega}$,  where $R_n$ is the restriction of $R_{\omega}$ to the constants plus $n$ extra elements. Each of those structure has only one sibling. 

If $R$ is universal for its universal theory $T_{\forall} (R)$, then it is equimorphic to a countable existentially universal structure, \cite{pouzet1972}. This structure,  unique up to isomorphism,  could play the role that uniform prehomogeneity plays in the case of $\aleph_0$ categoricity. 

If the conclusion of  Problem \ref{prob:extension} holds, a consequence is that if the profile of $R$ is finite, $R$ is universal and $Ker(R)$,   the  \emph{kernel} of $R$,  (the set  of $x\in E$ such that $\age (R_{-x})\not = \age(R)$) is infinite,  then  the number of siblings of $R$ is $2^{\aleph_0}$. Indeed, if $Ker(R)$ is infinite, $R$ cannot be finitely partitionable, nor cellular. So an obvious is to prove the following directly.

\begin{problem}\label{prob:extension2}
Let $R$ be a countable relational structure with finite profile and an infinite kernel. Prove that  if $R$ is universal for its age,  $\sib(R)=2^{\aleph_0}$. \end{problem}

As shown below, a  positive answer to Problems \ref{prob:extension} has  some consequences on the number of countable models of a universal theory;  one of which we know is true, the other conjectured. 

\subsection{Problems on the  number of countable models}\label{sub:countable}

Thomass\'e's conjecture is a specific question about the number of models of universal theories. As it is well known, there are complete theories with any $n, n\geq 3$, countable models, and  Ehrenfeucht's  families of examples provide such theories. Indeed set $R:=(\Q,~\leq,(c_n)_{n\in \N})$ where $c_n$ is the constant $n$; then the theory of $R$ contains, up to isomorphy  exactly three countable models: $R$, $R+\Q:= (\Q+\Q', \leq, (c_n)_{n\in \N})$, $R+\{a\}+\Q':= (\Q+\{a\}+ \Q, \leq, (c_n)_{n\in \N})$. The last two are equimorphic, but there  are  $2^{\aleph_0}$ equimorphic models. 

 It is possible that   Thomass\'e's conjecture holds for any countable relational structure and that the solution comes from  set theoretical or model theoretical techniques. Structures having $1$ or $\aleph_0$ siblings  must be exceptional and their description seems to be an interesting task. In that respect, we believe that the significant part of our result is the characterization of the structures $R$ such that $\sib(R)$ is $1$.  

Now let $\mathcal C$ be a hereditary class of finite relational structures in a finite language. Let $\overline {\mathcal C}_{\aleph_0} $ be the class of countable $R$ (up to isomorphy) such that $\age (R) \subseteq \mathcal C$, and let $\overline {\mathcal C}_{\aleph_0}/\equiv$ be the set of equimorphism classes of members of   $\overline {\mathcal C}_{\aleph_0}$. If further $\mathcal A$  is an age, let  $I({\mathcal A})$ be the number  of  countable $R$  (up to isomorphy) such that $\age (R)=\mathcal A$ and let $I({\mathcal A})/\equiv$ be the number  of equimorphism classes of countable $R$ such that $\age (R)= \mathcal A$. In this case the possible values of $I(\mathcal A)$ \label {pb-itemc} are known.   Indeed it was shown by Macpherson, Pouzet, and Woodrow in \cite{macpherson-pouzet-woodrow}   that $I(\mathcal A)$ is either $1$, $\aleph_0$ or $2^{\aleph_0}$. They further proved that $I(\mathcal A)$ is at most countable iff all $R$ with age $\mathcal A$ are cellular.  A positive answer to Problem \ref{prob:extension}  would be a generalization of the first part of their result; for consider two cases: \\
a) The number of maximal existential types which appear in those $R$ is uncountable. In this case the number of these types is $2^{\aleph_0}$, simply because these types form a $G_{\delta}$ set, and thus it follows that the number of isomorphic types of countable $R$ is $2^{\aleph_0}$. \\
b) This number is at most countable. In this case there is some countable $R'$ with  age $\mathcal A$ which is universal (see \cite{pouzet1972}), and thus we apply the conclusion of Problem \ref{prob:extension}.

The following problem remains.

\begin{problems}
\label {pb-itemd} Let $\mathcal A$ be an age, find the possible values of $I(\mathcal A)/\equiv$. 
\end{problems}

\noindent Concerning this problem, Pouzet, Sauer and Thomass\'e had  conjectured in 2006 that  $I(\mathcal A)/\equiv$   is either $1$, $\aleph_0$, $\aleph_1$ or $2^{\aleph_0}$; note that these values do occur as the age of an infinite path yields  $\aleph_{0}$, and the age of an infinite chain yields $\aleph_{1}$. In fact, as observed by Melleray in \cite{melleray} it follows from a theorem of Burgess (see \cite{srivastava}, Theorem 5.13.4 page 230) that the number is $2^{\aleph_0}$ whenever it is larger than $\aleph_1$.  Indeed,  the set of relational structures $R$ with age $\mathcal A$ and domain $\N$ is a $G_{\delta}$ set and the equivalence relation of  equimorphy is analytic.  One can say more. Let $\kappa:=  I(\mathcal A) / \equiv  < 2^{\aleph_0}$, then there is a countable universal structure, say $U$,  with age $\mathcal A$. Indeed, as mentioned above,  the number of maximal existential types which appear in the $R$ with age $\mathcal A$ is at most $\kappa$ (by  taking a representative in each equimorphy class and the maximal existential types appearing in that representative). Since these types form a $G_{\delta}$ set, their number in this case must be  countable. From the criteria given in  \cite{pouzet1972}, there is a a countable structure $U$ which is universal. The equivalence relation of equimorphy on subsets of $U$ is analytic  and we may now apply Burgess' result.  

These same three authors had also conjectured that $I(\mathcal A) / \equiv $ is $1$ iff all countable  $R$ with $\age(R)= \mathcal A$ are  cellular. This last conjecture is also somewhat related to Problem \ref{prob:extension}  since all countable $R$ are universal.  We do not know the answer to this simple question: Is $\sib(R)=2^{\aleph_0}$ and $\vert I(\mathcal A(R)) / \equiv $ equal to $1$ impossible?

Now consider two ages $\mathcal A \subseteq \mathcal A'$. It follows, from a result of Hodkinson and Macpherson and the result of  Macpherson, Pouzet, Woodrow \cite {macpherson-pouzet-woodrow} already mentioned, that $I(\mathcal A)\leq I(\mathcal A')$. Indeed  $I(\mathcal A)$ is either $1$, $\aleph_0$ or $2^{\aleph_0}$. Moreover,  $I(\mathcal A)$ is  $1$ iff some (in fact every) countable structure with age $\mathcal A$ is finitely partitionned (Hodkinson-Macpherson),  and $I(\mathcal A)$ is at most countable iff some (in fact every) countable structure with age $\mathcal A$ is cellular. 
Thus, either $I(\mathcal A')= 2^{\aleph_0}$, in which case $I(\mathcal A) \leq I(\mathcal A')$, or $I(\mathcal A')= \aleph_{0}$, in which case $\mathcal A'$ is the age of a countable cellular structure,  hence $\mathcal A$ too,  and thus $I(\mathcal A)\leq \aleph_0= I(\mathcal A')$, or else $I(\mathcal A')= 1$, in which case $\mathcal A'$ is the age of a finitely partitionnable structure, hence $\mathcal A$ too and thus $I(\mathcal A)=1= I(\mathcal A')$.

Thus once again the following problem naturally follows.

\begin{problems}
If  $\mathcal A \subseteq \mathcal A'$ are both ages, is $I(\mathcal A)/\equiv \; \leq \; I(\mathcal A')/\equiv$?
\end{problems}

The following particular case could shed some light. 

\begin{problems} Let $\mathcal A$ be an age, and suppose that  for every $n\in \N$ the collection of countable $R$ with $n$ unary relations added and such that $\age(R)=\mathcal A$   is w.q.o.  Does it follow that $I(\mathcal A)/\equiv$ is at most $\aleph_1$? \\
Conversely, if $I(\mathcal A)/\equiv$ is at most $\aleph_1$ ( and $\aleph_1<2^{\aleph_0}$), does it follows that the collection of countable $R$ such that $\age(R)= \mathcal A$ is w.q.o.? 
\end{problems}

\noindent We mentioned earlier that it is known (Laver 1971 \cite{laver}) that $I(\mathcal A)/\equiv$ is $\aleph_1$ if $\mathcal A$ is the age of an infinite chain. Is the same true  for the age of finite cographs? This is relevant here as countable relations with this age and $n$ unary relations added do form a w.q.o. (Thomass\'e \cite{thomasse1}).

A related result is the following. Let $R$ be  relational structure and $\varphi_R (\kappa)$ be the number of restrictions of $R$ to subsets  $A$ of size $\kappa$, these restrictions counted up to isomorphy. According to  Gibson, Pouzet,  Woodrow \cite{gibson} $\varphi_R (n)\leq  \varphi_R (\aleph_0)$ for $n<\omega$. And $\varphi_R (\aleph_0)$ can be finite, $\aleph_0$ or $2^{\aleph_0}$.

\begin{problem}
Let $R'$ be  a countable structure; it is well known that, for every countable $R$  with $\age(R)\subseteq  \mathcal A'= \age (R')$, there is a countable  extension of $R$  and $R'$ with age $\mathcal A'$. Consequently, the collection of countable $R'$ with a given age, say $\mathcal A'$,  is up-directed.  Is is true that  for every $R'$ in this set, the number of $R''$ above $R'$ is equal to the cardinality of this set?  
\end{problem}

If this is true, then the answer to Problem \ref {prob:extension}  is positive; indeed, if there is a universal $R$, the number of siblings will be  the number of structures with the same age.

We conclude with the following general problem. 

\begin{problems}\label{problem-heredit}
If $\mathcal C$ is a hereditary class, find the possible values of $\vert \overline {\mathcal C}_{\aleph_0}\vert$ and 
$\vert\overline {\mathcal C}_{\aleph_0}/\equiv \vert$.
\end{problems}


\begin{thebibliography}{10pt}
\bibitem{bonato-tardif} A.~Bonato, C.~Tardif,  Mutually embeddable graphs and the tree alternative conjecture. J. Combin. Theory Ser. B 96 (2006), no. 6, 874--880.
\bibitem{bonato-al} A.~Bonato, H.~Bruhn, R.~Diestel, P.~Spr\"ussel, Twins of rayless graphs. J. Combin. Theory Ser. B 101 (2011), no. 1, 60-65. 
\bibitem{boudabbous}Y.~Boudabous and M.~Pouzet, The morphology of infinite tournaments; application to the growth of their profile. European Journal of Combinatorics. 31 (2010) 461-481.
\bibitem{cameron} P.J.~Cameron, Transitivity of permutation groups on unordered sets. Math. Z. 148 (1976), no. 2, 127--139.
\bibitem{cameron.1990} P.J.~Cameron, Oligomorphic permutation groups.  London  Mathematical Society Lecture Note Series, volume 152. Cambridge University Press, Cambridge, 1990.
\bibitem{diestel} R.~Diestel, Graph theory. Fourth edition. Graduate Texts in Mathematics, 173. Springer, Heidelberg, 2010. xviii+437 pp.
\bibitem{fraisse-lopez 90} R.~Fra\"{\i}ss\'e, G.~Lopez, La reconstruction d'une relation dans l'hypoth\`ese forte: isomorphie des restrictions \`a chaque partie stricte de la base. With an appendix by Lopez and C. Rauzy. S\'eminaire de Math\'ematiques Sup\'erieures, 109. Presses de l'Universit\'e de Montr\'eal, Montr\'eal, QC, 1990. 139 pp. ISBN: 2-7606-1529-4.  

\bibitem{fraisse0}R.~Fra\"{\i}ss\'e, Sur l'extension aux relations de quelques propri\'et\'es des ordres, Ann. Sci. \'Ecole Norm. Sup. 71 (1954), 361-388. 
\bibitem{fraisse} R.~Fra{\"\i}ss{\'e}, Theory of relations. Second edition, North-Holland Publishing Co., Amsterdam, 2000.
\bibitem{frasnay 65} C.~Frasnay, Quelques probl\`emes combinatoires concernant les ordres totaux et les relations monomorphes. Th\`ese. Paris.  Annales Institut Fourier Grenoble  15 (1965), 415--524.
 \bibitem{frasnay 84} C.~Frasnay, Chainable relations, rangements and pseudorangements. Orders: description and roles (L'Arbresle, 1982), 235-268, North-Holland Math. Stud., 99, North-Holland, Amsterdam, 1984. 
 \bibitem{frasnay 90}C.~Frasnay, D\'etermination du degr\'e optimal $d_m$ de monomorphie pour les structures relationnelles au plus $m$-aires. Math. Rep. Acad. Sci. Canada 12 (1990) no. 4, 141--146.
\bibitem{gibson} P.C.~Gibson, M.~Pouzet, R.E.~Woodrow, Relational structures having finitely many full-cardinality restrictions. Discrete Math. 291 (2005), no. 1-3, 115--134.
\bibitem{hodkinson-macpherson} I.M.~Hodkinson, H.D.~Macpherson,  Relational structures determined by their finite induced substructures. J. Symbolic Logic 53 (1988), no. 1, 222--230.
\bibitem{ille 92} P.~Ille, The reconstruction of multirelations, at least one component of which is a chain. J. Combin. Theory Ser. A 61 (1992), no. 2, 279--291.

\bibitem{lachlan-woodrow}A.H.~Lachlan, R.E.~Woodrow, Countable ultrahomogeneous undirected graphs.
Trans. Amer. Math. Soc. 262 (1980), no. 1, 51?94. 
\bibitem{LPW} C.~Laflamme, M.~Pouzet, R.~Woodrow, Equimorphy- The case of chains. Archive for Mathematical Logic. Arch. Math. Logic 56 (2017), no. 7-8, 811-829.
%
\bibitem{lopez 72}G.~Lopez, Sur la d\'etermination d'une relation par les types d'isomorphie de ses restrictions. C. R. Acad. Sci. Paris Ser. A-B 275 (1972) A951--A953.
%
\bibitem{lopez 78}G.~Lopez, L'ind\'eformabilit\'e des relations et multirelations binaires. Z. Math. Logik Grundlag. Math. 24 (1978), no. 4, 303--317.
%
\bibitem{laver} R.~Laver, On Fra\"{\i}ss\'e's order type conjecture,  
Ann. of Math. (2) 93  (1971) 89?111.
\bibitem {macpherson-pouzet-woodrow} H.D.~Macpherson, M.~Pouzet, R.E.~Woodrow, Countable structures of given age. J. Symbolic Logic 57 (1992), no. 3, 992--1010.
\bibitem{melleray} J.~Melleray, Personnal communication, december 2017. 
\bibitem{oudrar-pouzet} D.~Oudrar, M.~Pouzet, D\'ecomposition monomorphe des structures relationnelles et profil de classes h\'er\'editaires, sept. 2014, 7p. arXiv:1409.1432 Working document, April 2014. 
\bibitem{oudrar} D.~Oudrar, Sur l'\'enum\'eration de structures discr\`etes, une approche par la th\'eorie des relations,  Th\`ese de doctorat, Universit\'e d'Alger  USTHB  \`a Bab Ezzouar,  28 sept. 2015, 249p.,  arXiv:1604.05839
\bibitem{pabion}J-F. ~Pabion, Relations pr\'ehomog\`enes.  C. R. Acad. Sci. Paris S\'er. A-B 274 (1972), A529--A531.
\bibitem{pouzet1972} M.~Pouzet, Mod\`ele universel d'une th\'eorie $n$-compl\`ete.  C. R. Acad. Sci. Paris S\'er. A-B 274 (1972), A433--A436.
\bibitem{pouzet}M.~Pouzet, Mod\`ele universel d'une th\'eorie $n$-compl\`ete: Mod\`ele uniform\'ement pr\'ehomog\`ene. C. R. Acad. Sci. Paris S\'er. A-B 274 (1972), A695--A698.

\bibitem{pouzet2}M.~Pouzet,Mod\`ele universel d'une th\'eorie $n$-compl\`ete: Mod\`ele pr\'ehomog\`ene.  C. R. Acad. Sci. Paris S\'er. A-B 274 (1972), A813--A816. 

\bibitem{pouzettr} M.~Pouzet, Sur la th\'eorie des relations.  Th\`ese d'\'etat, Universit\'e Claude-Bernard, Lyon 1, pp.~78--85, 1978.
\bibitem{pouzet79} M.~Pouzet, Relations non reconstructibles par leurs restrictions. J. Combin. Theory Ser. B 26 (1979), no. 1, 22--34. 
\bibitem{pouzet79} M.~Pouzet, Relation minimale pour son \^age. Z. Math. Logik Grundlag. Math. 25 (1979), no. 4, 315--344.
\bibitem{pouzet 81} M.~Pouzet,  Application de la notion de relation presque-encha\^{\i}nable au   d\'enombrement des restrictions finies d'une relation. Z. Math. Logik Grundlag. Math. 27(4) (1981), 289--332.
 \bibitem{pouzet-roux}M.~Pouzet, B.~Roux, Ubiquity in category for metric spaces and transition systems. Discrete metric spaces (Bielefeld, 1994). European J. Combin. 17 (1996), no. 2-3, 291--307.
\bibitem{sikaddour-pouzet}  M.~Pouzet, H.~Si-Kaddour, Isomorphy up to complementation, Journal of Combinatorics, 7 (2016), no 2, 285-305.
\bibitem{pouzet-thiery} M.~Pouzet, N.~Thi\'ery, Some relational structures with polynomial growth and their associated algebras I. Quasi-polynomiality  of the profile, The Electronic J. of Combinatorics, 20(2) (2013), 35pp.
\bibitem{robinson}A.~Robinson, 
Forcing in model theory. Actes du Congr\`es International des Math\'ematiciens (Nice, 1970), Tome 1, pp. 245-250. Gauthier-Villars, Paris, 1971. 


\bibitem{saracino}D.~Saracino,  Model companions for $\aleph_0$-categorical theories. Proc. Amer. Math. Soc. 39 (1973), 591--598. 
\bibitem{schmerl} J.~Schmerl, Coinductive $\aleph_0$-categorical theories. J. Symbolic Logic, 55 (1990), no. 3, 1130–1137. 
\bibitem{schreierulam} J.~Schreier, S~Ulam, \"{U}ber die Permutationsgruppe der Nat\"{u}rlichen Zahlen-folge, Studio Math., 4 (1933), 134-141.
\bibitem{scott} W.~.R.~Scott, Group Theory, Second edition. Dover Publications, Inc., New York, 1987.
\bibitem{simmons} H.~Simmons, Large and small existentially closed structures. J. Symbolic Logic, 41 (1976), no. 2, 379-390.
\bibitem{srivastava} S.M.~Srivastava, A course on Borel sets. Graduate Texts in Mathematics, 180. Springer-Verlag, New York, 1998. 
\bibitem{thomasse1} S.~Thomass\'e, On better quasi ordering countable series-parallel orders, Trans. Amer. Math Soc. 352 (6) (1999) 2491--2505.
\bibitem{thomasse} S.~Thomass\'e, Conjectures on Countable Relations, circulating manuscript, 17p.  2000, and personal communication, November 2012. 
\bibitem{todorcevic}I.~Farah and S.~Todorcevic, Some applications of the method of forcing, Yenisei series in pure and applied mathematics, Yenisei, Moscow, 1998.
\bibitem{tyomkyn} M.~Tyomkyn, A proof of the rooted tree alternative conjecture, Discrete Math. 309 (2009) 5963-5967.
\bibitem{vuksanovic} V.~Vuksanovic, A proof of a partition theorem for $[\Q]^n$, Proceedings of AMS  130 (2002), No. 10, 2857-2864.
\end{thebibliography}
\end{document}